\documentclass[12pt]{article}

\usepackage[utf8x]{inputenc}
\usepackage{amsmath}
\usepackage{amssymb}
\usepackage{amsfonts}
\usepackage{amsthm}
\usepackage{mathrsfs}
\usepackage{bbm}
\usepackage[all]{xy}

\def\|{\mathbb}
\def\@{\mathcal}
\def\!{\mathscr}

\def\~{\widetilde}
\def\-{\overline}
\def\Aff{\operatorname{Aff}}
\def\Spec{\operatorname{Spec}}
\def\Sets{\operatorname{\mathbf{Sets}}}
\def\gr{\operatorname{gr}}

\def\ord{\operatorname{ord}}
\def\ac{\operatorname{\overline{ac}}}
\def\Ldp{\mathcal{L}_{\mathrm{LD,P}}}
\def\Lval{\mathbf{L}_{\mathrm{Val}}}
\def\Lres{\mathbf{L}_{\mathrm{Res}}}
\def\Lord{\mathbf{L}_{\mathrm{Ord}}}
\def\Val{\mathrm{Val}}
\def\Res{\mathrm{Res}}
\def\Ord{\mathrm{Ord}}
\def\Def{\mathrm{Def}}
\def\RDef{\mathrm{RDef}}

\def\Field{\mathrm{Field}}

\newtheorem{theorem}{Theorem}[section]
\newtheorem{claim}[theorem]{Claim}
\newtheorem{corollary}[theorem]{Corollary}
\newtheorem{definition}[theorem]{Definition}
\newtheorem{lemma}[theorem]{Lemma}
\newtheorem{proposition}[theorem]{Proposition}

\newenvironment{convention}[1][Convention]{\refstepcounter{theorem} \vspace{6pt}\noindent \textbf{#1 \arabic{section}.\arabic{theorem}. }}
{

\vspace{6pt}}

\newenvironment{remark}[1][Remark]{\refstepcounter{theorem} \vspace{6pt}\noindent \textbf{#1 \arabic{section}.\arabic{theorem}.}}{

\vspace{6pt}}

\title{$p$-adic and motivic measure on Artin $n$-stacks}
\author{Chetan Balwe}
\date{}

\begin{document}

\maketitle

\section{Introduction}

Let $X$ be a scheme of finite type and pure dimension $d$ over $\|Z_p$. One can define a measure on the space $X(\|Z_p)$, called the $p$-adic measure which we denote by $\mu_d$. Roughly speaking, this is defined by choosing bi-analytic isometries of open subsets of the smooth part of $X(\|Z_p)$ with balls in $\|Z_p^d$ and pulling back the normalized Haar measure on $\|Z_p^d$ (\cite{Se2}). However, there is another way to define this measure. For each $n$, let $\tau_n: X(\|Z_p) \rightarrow X(\|Z/p^n \|Z)$ be the ``reduction modulo $p^n$" map. Let $A$ be a sub-analytic or definable (in the language of valued fields) subset of $X(\|Z_p)$. Then it can be proved (see \cite{Os}) that 
\[
\mu_d(A) = \lim_{n \rightarrow \infty} p^{-nd} \tau_n(A) \text{.}
\]
In other words, the $p$-adic measure on $X(\|Z_p)$ can be obtained from the counting measure on $X(\|Z/p^n\|Z)$ by a limiting process. As an application of $p$-adic measure, it can be proved that the power series
\begin{eqnarray*}
P_X(T):= \sum_{n=0}^{\infty} |\tau_n(X(\|Z_p))| T^n  \text{,} & \~P_X(T):= \sum_{n=0}^{\infty} |X(\|Z/p^n \|Z)| T^n
\end{eqnarray*}
are rational functions of $T$ (see \cite{De}). 

If $X$ is a scheme of finite type and pure dimension $d$ over $\|Z$, one can consider the schemes $X_p := X \times_{\Spec(\|Z)} \Spec(\|Z_p)$ for various primes $p$. Then motivic integration allows us to compare the $p$-adic measures on $X_p$ as $p$ varies. Roughly speaking, if we consider a formula $\phi$ in the language of valued fields and interpret it on the various $X_p$, we obtain a family of subsets $A_p \subset X(\|Z_p)$ for almost all $p$. The motivic measure $\mu(\phi)$ of the formula $\phi$ lies in a certain localization of the Grothendieck ring of formulas in the language of rings with coefficients in $\|Z$. Then for almost all $p$, the measure $\mu_d(A_p)$ can be obtained from $\mu(\phi)$ by a process of specialization which amounts to ``counting the $\|F_p$-valued points satisfying $\mu(\phi)$" (see \cite{CL2}, Section 9 for a precise discussion). As a result, for almost all $p$, evaluating the $p$-adic measure of a definable subset for any fixed $p$ boils down to counting the $\|F_p$-valued points satisfying a set of formulas which is independent of $p$. Thus, one is able to strengthen the above-mentioned result regarding the rationality of the power series $P_{X_p}$ and $\~P_{X_p}$. Indeed, one is able to obtain rational functions $P_X(T)$ and $\~P_X(T)$ in $T$ with coefficients in the above-mentioned localization of the Grothendieck ring which specialize to the power series $P_{X_p}(T)$ and $\~P_{X_p}(T)$ respectively, for almost all $p$. 

We would like to generalize these results to Artin stacks which are strongly of finite type over $\|Z$. First we would like to define $p$-adic measure on an Artin stack $X$ which is strongly of finite type over $\|Z_p$. We do this by first defining a counting measure for the $(\|Z/p^n \|Z)$-valued points of $X$. Then we use this counting measure and obtain $p$-adic measure by a limiting process as we indicated above in the case of varieties. As a consequence of the proof, we will see that the power series $P_X(T)$ and $\~P_{X}(T)$ are rational functions of $T$. Finally, when $X$ is an Artin stack which is strongly of finite type over $\|Z$, we will explore the notion of motivic measure for Artin stacks and obtain a uniform rationality result for the power series $P_{X_p}$. 

\paragraph{Acknowledgments:}
I would like to thank François Loeser for his encouragement and for giving me numerous helpful suggestions on this topic. I was supported by Fondation Sciences Mathématiques de Paris through a post-doctoral fellowship and I would like to thank them for their generous support during the completion of this work.

\begin{convention}
We will use the following conventions and notations: 
\begin{itemize}
\item[(1)] For any scheme $T$, $\Aff/T$ will denote the (big) étale site of affine schemes over $T$. 
\item[(2)] For any scheme $T$, $(\Aff/T)^{\sim}$ will denote the model category of simplicial presheaves on $T$ with the \textit{local projective model structure} (see \cite{TV1}). The homotopy category $Ho((\Aff/T)^{\sim})$ will be referred to as the category of stacks over $T$ and denoted by $St(T)$.  
\item[(3)] We will usually be concerned with Artin stacks that are strongly of finite type over the base scheme. For the sake of brevity, we will say that $X$ is an sft-Artin stack over $S$ if $X$ is an Artin stack, strongly of finite type over $S$ (see \cite{TV2} for the definitions). 
\item[(4)] When we speak of sft-Artin stacks over a discrete valuation ring $A$, we will always intend it to be \textit{flat} over $A$.
\end{itemize}
\end{convention}

\section{Counting points on Artin stacks}
\label{Greenberg}

In this section, we consider the problem of meaningfully defining a notion of counting the$(\|Z/p^n\|Z)$-valued points of a sft-Artin stack over $\|Z_p$. In other words, for such a stack $X$, we wish to define a counting measure on the set $\pi_0(X(\|Z/p^n \|Z))$. 

If $X$ is an sft-Artin stack over a finite field $\|F_q$ there is already a notion of counting the $\|F_q$-valued points on $X$ (see \cite{To}, Prop. 3.5). Indeed, one defines 
\begin{equation}
\label{toen-counting}
\#X(\|F_q):= \sum_{x \in \pi_0(X(\|F_q))} \prod_{i>0} |\pi_i(X(\|F_q),x)|^{{(-1)}^i}
\end{equation}
where $|\cdot|$ denotes the cardinality of a set. (Both the sum and the product in the above expression are finite.) This definition is justified by the fact that this counting function factors through the Grothendieck ring of sft-Artin stacks over $\|F_q$.

The above formula suggests that if one has an sft-Artin stack $X$ over $\|Z_p$, one may wish to simply define
\begin{equation}
\label{R_n-counting}
\#X(\|Z/p^n \|Z) := \sum_{x \in \pi_0(X(\|Z/p^n\|Z))} \prod_{i>0} |\pi_i(X(\|Z/p^n\|Z),x)|^{{(-1)}^i} 
\end{equation}
(assuming that the expression on the right-hand side is finite). Indeed, this is what we will do, but this is not merely an ad hoc definition. One can functorially construct an Artin stack $Gr_n(X)$ of strongly finite type over $\|F_p$ such that we have a weak equivalence $X(\|Z/p^n\|Z) \simeq Gr_n(X)(\|F_p)$ and so that applying the formula (\ref{toen-counting}) to $Gr_n(X)$ yields the formula (\ref{R_n-counting}). 

More generally, let $R$ be a complete discrete valuation ring with residue field $k$. Let $\omega$ be a uniformizing parameter of $R$ and let $R_n:= R/<\omega^{n+1}>$ for $n \geq 0$. Then, given a sft-Artin stack $X$ over $R$, we will construct sft-Artin stacks $Gr_n(X)$ over $k$ such that $X(R) \simeq Gr_n(X)(k)$.

We begin by recalling material from \cite{Gr}. Suppose $S$ is an Artin local ring with residue field $K$ such that there exists a bijection of the elements of $S$ with the set $K^n$ for some $n$ and such that the addition and multiplication maps are given by polynomials with coefficients in $K$. Then there exists a ring variety $\@S$ over $K$, whose underlying scheme is $\|A^n_K$, and such that $\@S(K)=S$. We can use the ring variety $\@S$ to define a functor $\Aff/\Spec(K) \rightarrow \Aff/\Spec(S)$ given by $U \mapsto \Spec(\@S(A))$.

\begin{convention}
For any ring scheme $\@A$ over a base scheme $T$, we denote by $\~{\@A}: \Aff/T \rightarrow \Aff/\Spec(\@A(T))$ the functor $U \mapsto \Spec(\@A(U))$. 
\end{convention} 

It is proved in \cite{Gr}, that if $X$ is a scheme of finite type over $S$, then the presheaf on $\Aff/\Spec(K)$ defined given by $U \mapsto X(\~{\@S}(U))$ is represented by a scheme of finite type over $K$. In the case when $S$ is of the form $S=R_n$ for some $n \geq 0$ and we wish to extend this result to sft-Artin stacks over $R_n$.  

\begin{convention}
Let $C$ and $D$ be Grothendieck sites and $\sigma: C \rightarrow D$ be a functor. Then we use $(\sigma_!, \sigma_{*})$ to denote the adjunction $C^{\sim} \rightleftarrows D^{\sim}$ where $\sigma_*(F)$ is defined by $\sigma_*(F)(c) = F(\sigma(c))$. (Of course, this is not always a Quillen adjunction.)
\end{convention}

Now for each $n \geq 0$, let $\@R_n$ denote the ring variety over $k$ which is constructed from $R_n$ in the manner described above. We wish to examine the functors $(\~{\@R_n})_*$ for each $n \geq 0$. We will prove that the adjunction $((\~{\@R_n})_{!}, (\~{\@R_n})_*)$ is a Quillen adjunction and that the right derived functor of $(\~{\@R_n})_*$ (which will be the desired functor $Gr_n$) takes sft-Artin stacks over $R_n$ to sft-Artin stacks over $k$. 

The functor $\~{\@R_n}$ behaves differently depending on whether characteristics of $R$ and $k$. We will only need the case in which $char(R) \neq char(k)$. The equal characteristic case will is only mentioned in this section for the sake of completeness. Also, this case is much easier to handle since, in this case we know that $R$ is isomorphic to the power series ring $k[[t]]$. It is then easily verified that for $U = \Spec(A)$ in $\Aff/\Spec(k)$, we have $\~{\@R_n}(U) = U \times_{\Spec(k)} \Spec(R_n)$. Thus $(\~{\@R}_n)_*$ is simply Weil restriction with respect to the morphism $\Spec(R_n) \rightarrow \Spec(k)$ and we can define $Gr_n(X)$ to be the stack $\underline{Hom}(\Spec(R_n), X)$ where $\underline{Hom}$ denotes the internal Hom in $St(k)$ (\cite{TV1}, Section 3.6). However, we do not have this option when $char(R) \neq char(k)$ and thus, for the sake of a unified presentation, we simply follow the argument described in the preceding paragraph for both cases and prove that $((\~{\@R_n})_{!}, (\~{\@R_n})_*)$ is a Quillen adjunction. For the equal characteristic case, this is easy to see (and this was presented explicitly in \cite{Ba}). However, in the unequal characteristic case, a little more work is required as we will see in Prop. \ref{preserving-etale} and Prop. \ref{preserving-fibre-products}. 

We will now focus on the structure of the ring varieties $\@R_n$ when $char(R) = 0$, $char(k)=p \neq 0$. In this case, $R$ is obtained as a totally ramified extension of the ring $W(k)$ of Witt vectors with coefficients in $k$. Let us recall some basic facts about Witt vectors (see \cite{Ill} or \cite{Se}). For any $\|F_p$-algebra $A$, $W(A)$ is actually the set of $A$-rational points of the ring scheme of Witt vectors, denoted by $W$. The underlying scheme of $W$ is $\|A^{\|N} = \Spec \|F_p[Z_1, Z_2, \ldots]$. (Strictly speaking, this is only the fibre of the Witt scheme at the prime $p$. However, since we are only going to be working with $\|F_p$-algebras, this will suffice for our purposes.) The addition and multiplication are given by polynomials with coefficients in $\|F_p$. As it turns out, these polynomials, when restricted to the first $n$ coordinates, define a ring scheme structure on $\|A^n$ which is denoted by $W_n$ and called the scheme of Witt vectors of length $n$. For $n>m$, the projection on the first $m$ coordinates defines a ``truncation" morphism $W_n \rightarrow W_m$ which is a ring scheme homomorphism. Similarly, we have morphisms $W \rightarrow W_n$ and $W$ is the projective limit of the system defined by the $W_n$ for $n \geq 0$ along with the truncation morphisms. 

The scheme $W$ has two automorphisms - the Verschiebung or ``shifting´´ operator $V$ and the Frobenius operator $F$. Via the isomorphism $W \cong \|A^{\|N}$, these automorphisms are given on $W(A)$ for any $k$-algebra $A$ by the formulas
\[
V((a_0, a_1, \ldots)) = (0, a_0, a_1, \ldots)
\] 
and
\[
F((a_0, a_1, \ldots)) = (a_0^p, a_1^p, \ldots) 
\]
where $a_i \in A$, $\forall i$. In other words, $F$ is just induced by the Frobenius operator on $A$ (which we will denote by the same symbol $F$). (Note that this description of the Frobenius operator only applies when we are working with algebras over $\|F_p$. For more general rings, the description is via the ``ghost components" of the Witt vectors.)  

We will require the following easily verifiable facts about these operators:
\begin{itemize}
\item[(1)] For any $a,b \in W(A)$, we have $aV(b) = V(F(a)b)$. 
\item[(2)] Iterations of $V$ induce a filtration of $W$ that is consistent with the ring structure. In other words $V^{n}W(A) \cdot V^{m}W(A) \subset V^{n+m}W(A)$. Note that $V^n W(A)$ is the kernel of the truncation map $W(A) \rightarrow W_n(A)$. 
\end{itemize}
Note that we have similar operators $V$ and $F$ on $W_n(A)$ for any $n$ and that these maps commute with the truncation map. (It is also an easily verifiable fact that $VF$ and $FV$ are equal to multiplication by $p$.)   

Let $\gr_V W(A)$ denote the graded ring of $W(A)$ with respect to the filtration induced by $V$. Then it follows from statement (1) that $\gr^n_V W(A)$ is isomorphic to the $A$-module $F^n_{*}A$ obtained by considering $A$ as an $A$-module by scalar restriction via the homomorphism $F^n: A \rightarrow A$. 

Now, let us consider the rings $R$ and $R_n$ for $n \geq 0$. We know that $R$ is obtained from $W(k)$, by attaching the root of an Eisenstein polynomial of degree $e$ and coefficients in $W(k)$ where $e$ is the absolute ramification index of $R$ over $W(R_0)$ (i.e. $ord_R(p) = e$). Thus $R$ is a free module with $e$ generators over the ring $W(R_0)$.  From this, it is easy to see that there exists a ring scheme $\@R$ which is a module scheme over $W$ of the form $W^e$ and such that $R = \@R(k)$. The rings $R_n$ are Artin local rings and so by the above discussion, we see that there exists ring schemes $\@R_n$ of finite type over $k$, such that such that $\@R_n(k) = R_n$. (Remark: The last equality is generalized in \cite{BLR} as $\@R_n(A) =  R_n \otimes_{W(k)} W(A)$ where $A$ is \textit{any} algebra over $k$. However, in reality, this will not hold unless $A$ is perfect.) We also have ring homomorphisms homomorphisms $\@R_n \rightarrow \@R_m$ for $n>m$ obtained from the surjections $R_n \rightarrow R_m$ and $\@R$ is the limit of this projective system of schemes.

We wish to prove that the functor $\@R_n$ preserves étale morphisms of schemes. This has been proved for the case $R = W(k)$ in \cite{Ill}. The proof in the general case can be obtained by the same method with some modifications. The proof is presented here in full detail for the sake of completeness.

For any integer $m \geq 1$, let $\@R^m$ (resp. $\@R^m_n$) denote the kernel of the morphism $\@R \rightarrow \@R_{m-1}$ (resp. $\@R_n \rightarrow \@R_{m-1}$). Let $\@R^0$ (resp. $\@R_n^0$) be the scheme $\@R$ (resp. $\@R_n$). Then $\{ \@R^m \}_{m \geq 0}$ (resp. $\{ \@R^m_n \}_{m \geq 0}$ ) is a decreasing filtration of closed subschemes on $\@R$ (resp. $\@R_n$). 

The ideals $\@R^m$ and $\@R^m_n$ can be easily described if we choose a good presentation of $\@R$ and $\@R_n$ as $W$-modules. We know that $R = W(R_0)[\omega]$ (see \cite{Se}, Chapter I, Prop 18). Thus we may choose $1, \omega, \ldots, \omega^{e-1}$ as generators for $R$ and $R_n$ as $W(k)$-modules. A typical element of $R$ is of the form 
\[
x = \sum_{i=0}^{e-1} (a_{0i}, a_{1i}, \ldots) \cdot \omega^{i}
\]
where $(a_{0i}, a_{1i}, \ldots) \in W(k)$. Clearly, 
\[
ord_R((a_{0i}, a_{1i}, \ldots) \cdot \omega^{i}) = i + ke
\]
where $k$ is the least integer such that $a_{ki} \neq 0$. Thus 
\[
ord_R(x) = \min_i\{ord_R(a_{0i}, a_{1i}, \ldots) \cdot \omega^{i})\} \text{.}
\] 
Suppose that $m+1 = q \cdot e + r$ for $0 \leq r \leq e-1$. Then it is easy to see from the above calculation of $ord_R(x)$ that  $\@R^m = V^{m_0}W \oplus \ldots V^{m_{(e-1)}}W$ where $m_i = q + \epsilon_i$ where $\epsilon_i = 1$ if $0 \leq i < r$ and $=0$ otherwise. (The description of $\@R^m_n$ is similar.) From this description, it is easy to see that the filtration is consistent with the ring structure and that the ideals $\@R^1_n$ are nilpotent for all $n$. 

Let $\gr \@R(A)$ (resp. $\gr \@R_n(A)$) denote the graded ring of $\@R(A)$ (resp. $\@R_n(A)$) with respect to this filtration. It follows that for $m<n$,  
\[
\gr^m \@R(A) = \gr^m \@R_n(A)= 
\begin{cases}
F^{q}_* A & \text{if $r \neq 0$} \\
F^{q+1}_* A  & \text{if $r =0$}
\end{cases}
\]
where $q$ and $r$ are as in the preceding paragraph. 

Now we prove that the functor $\~{\@R_n}$ takes étale maps into étale maps. In the special case $\@R=W$, this has been proved in \cite{Ill}. (An extension of this result involving general Witt vectors is proved in \cite{Bo}). Our argument for general $\@R$ is an adaptation of the proof in \cite{Ill}, but we present it in detail for the sake of completeness. 

\begin{proposition}
\label{preserving-etale}
Let $X \rightarrow Y$ be an étale morphism of schemes over $k$. Then $\~{\@R_n}(X) \rightarrow \~{\@R_n}(Y)$ is étale and the diagram
\[
\xymatrix{
X \ar[r] \ar[d] & \~{\@R_n}(X) \ar[d] \\
Y \ar[r] & \~{\@R_n}(Y)
}
\]
is cartesian.
\end{proposition}

\begin{proof} 
In the case $char(R) = char(k)$, the statement is obvious. Thus we now focus on the case $char(R) = 0$, $char(k) = p \neq 0$. Without any loss of generality, we may assume that $X = \Spec(B)$, $Y = \Spec(A)$ and the morphism $X \rightarrow Y$ is given by a $k$-algebra homomorphism $A \rightarrow B$. First we show that $\@R_n(A) \rightarrow \@R_n(B)$ is flat. For this we use a modification of the flatness criterion for filtered modules in \cite{Bour}, Chap. III, \S 5, Thm. 1. This criterion is stated there for $I$-adic filtrations but the arguments are easily adapted to this case. For the sake of completeness, the result is stated in the following lemma:

\begin{lemma}
\label{flatness-criterion}
Let $A$ be a ring with a given decreasing filtration $\{A_i\}_{i=0}^{\infty}$. Let $M$ be an $A$-module with a filtration $\{M_i\}_{i=0}^{\infty}$ which is compatible with the filtration of $A$. Suppose the following conditions hold:
\begin{itemize}
\item[(1)] There exists an integer $k$ such that $A_i = 0$ and $M_i = 0$ for all $i > k$. 
\item[(2)] $M/M_1$ is a flat $A/A_1$-module.
\item[(3)] $\gr^n A \otimes_{\gr^0 A} \gr^0 M \rightarrow \gr^n M$ is an isomorphism.  
\end{itemize}
Then $M$ is a flat $A$-module.
\end{lemma}
\begin{proof}
First we prove that $A_n \otimes_A M \rightarrow M_n$ is a surjection for all $n$. Indeed, this is trivially true for $n >k$. For general $n$,
we consider the diagram
\[
\xymatrix{
& A_{n+1}M \ar[r] \ar[d] & A_nM \ar[r] \ar[d] & A_nM/A_{n+1}M \ar[r] \ar[d]& 0 \\
0 \ar[r] & M_{n+1} \ar[r] & M_n \ar[r] & \gr^n M \ar[r] & 0 \text{.}
}
\]
From the diagram it is clear that it will suffice to prove that the right vertical map is surjective (since then we can apply decreasing induction on $n$). The image of $A_nM/A_{n+1}M$ in $\gr^n M$ is the same as the image of 
\[
A_n/A_{n+1} \otimes_{A} M \rightarrow \gr^n M \text{.}
\] 
But $A_n/A_{n+1}$ is annihilated by $A_1$. So 
\[
A_n/A_{n+1} \otimes_{A} M \cong A_n/A_{n+1} \otimes_{A/A_1} M/(A_1M) \text{.}
\] 
Since $A_1M \subset M_1$, we have a surjection 
\[
A_n/A_{n+1} \otimes_{A/A_1} M/(A_1M) \rightarrow A_n/A_{n+1} \otimes_{A/A_1} M/M_1 \text{.}
\] 
But $A_n/A_{n+1} \otimes_{A/A_1} M/M_1 \rightarrow \gr^n M$ is an isomorphism by hypothesis. Thus we see $A_n \otimes_A M \rightarrow M_n$ is a surjection. In other words, $A_nM = M_n$ for all $n$.  

Now we claim that $A_n \otimes_A M \rightarrow M_n$ is an injection as well. Indeed, consider the diagram
\[
\xymatrix{
& A_{n+1} \otimes_A M \ar[r] \ar[d] & A_n \otimes_A M \ar[r] \ar[d] & (A_n/A_{n+1}) \otimes_A M \ar[r] \ar[d]& 0 \\
0 \ar[r] & M_{n+1} \ar[r] & M_n \ar[r] & \gr^n M \ar[r] & 0 \text{.}
}
\]
Again, using decreasing induction on $n$, we see that it is enough to show that the right vertical map is an isomorphism. For this, we observe that 
\[
(A_n/A_{n+1}) \otimes_A M \cong (A_n/A_{n+1}) \otimes_{A/A_1} M/A_1 M \cong (A_n/A_{n+1}) \otimes_{A/A_1} M/M_1 \text{.}
\]
Thus the right vertical map is an isomorphism by hypothesis. Thus, $A_n \otimes_A M \rightarrow M_n$ is an isomorphism for all $n$. 

Now we may forget about the given filtrations on $A$ and $M$ and apply \cite{Bour}, Chap. III, \S 5, Thm. 1 for the $A_1$-adic filtrations. In other words, the facts that $M/M_1 = M/A_1M$ is a flat $A/A_1$ module and $A_1 \otimes M \rightarrow A_1M$ is a bijection imply that $M$ is a flat $A$-module.
\end{proof}

We continue the proof of Prop. \ref{preserving-etale}:
By [SGA-5], XIV, \S1, Prop. 2, the relative Frobenius map $F_{\Spec(B)/\Spec(A)}$ is surjective and radicial. Since it is also étale, it is an isomorphism. Thus $F_{*}B \cong B \otimes_A F_{*} A$, and by iteration $F^q_{*}B \cong B \otimes_A F^q_{*} A$ for any positive integer $q$. Thus, by our earlier arguments $\gr^n \@R_n(B) \cong B \otimes_A \gr^n \@R_n(A)$. By lemma \ref{flatness-criterion}, we see that $\@R_n(B)$ is flat over $\@R_n(A)$. Also, from the proof of the previous lemma, we see that $\@R^1_n(B)$ is generated by the image of $\@R^1_n(A)$. Thus
\[
\xymatrix{
X \ar[r] \ar[d] & \~{\@R_n}(X) \ar[d] \\
Y \ar[r] & \~{\@R_n}(Y) 
}
\]
is cartesian and the vertical arrows are flat morphisms. Now the result follows from lemma \ref{flat-pullback-etale}
\end{proof}

\begin{lemma}
\label{flat-pullback-etale}
Suppose $A \rightarrow B$ is a flat ring homomorphism. Let $I$ be a nilpotent ideal in $A$ and suppose that $A/I \rightarrow B/IB$ is étale. Then $A \rightarrow B$ is étale. 
\end{lemma}
\begin{proof}
Since $I$ is a nilpotent ideal, one can easily prove that the fact that $A/I \rightarrow B/IB$ is of finite type implies that $A \rightarrow B$ is of finite type. Indeed if $p: A/I[X_1, \ldots, X_r] \rightarrow B/IB$ is a surjection, we define a morphism $q: A[X_1, \ldots, X_r] \rightarrow B$ to be an arbitrary lift of this surjection. Then for any element $b$, there exists $f(X) \in A[X_1, \ldots, X_r]$ such that $q(f(X)) - b \in IB$. Thus $q(f(X) - b = i_1 b_1 + \ldots + i_s b_s$. Now choose $f_i(X) \in A[X_1, \ldots, X_r]$ for $1 \leq i \leq s$ such that $q(f_i(X)) - b_i \in IB$. Then if $g(X) = f(X) - \sum_i f_i(X)$, we see that $q(g(X)) - b \in I^2 B$. Continuing in this manner and using the fact that $I$ is nilpotent, we see that $q$ is surjective. 

Thus now we merely need to prove that $A \rightarrow B$ is unramified. But this is immediate since $A/I \rightarrow B/IB$ is unramified. (A morphism of schemes is unramified if and only if its geometric fibres are unramified).
\end{proof}

\begin{proposition}
\label{preserving-fibre-products}
The functor $\~{\@R_n}: \Aff/\Spec(k) \rightarrow \Aff/\Spec(R_n)$ satisfies $\~{\@R_n}(X \times_Z Y) \cong \~{\@R_n}(X) \times_{\~{\@R_n}(Z)} \~{\@R_n}Y$ if $X \rightarrow Z$ is étale. 
\end{proposition}
\begin{proof}
In the diagram
\[
\xymatrix{
X \times_Z Y \ar[r] \ar[d] & X \ar[r] \ar[d] & \~{\@R_n}(X) \ar[d]\\
Y \ar[r] & Z \ar[r] & \~{\@R_n}(Z) \text{,}
}
\]
the left and right squares are cartesian and so the outer square is cartesian. In the diagram 
\[
\xymatrix{
X \times_Z Y \ar[r] \ar[d] & \~{\@R_n}(X) \times_{\~{\@R_n}(Z)} \~{\@R_n}(Y) \ar[r] \ar[d] & \~{\@R_n}(X) \ar[d] \\
Y \ar[r] & \~{\@R_n}(Y) \ar[r] & \~{\@R_n}(Z) \text{,}
}
\]
the right and outer squares are cartesian. Thus the left square is cartesian. But since $Y$ is a closed subscheme of $\~{\@R_n}(Y)$ defined by a nilpotent ideal, there is a unique étale morphism $T \rightarrow \~{\@R_n}(Y)$ such that $T \times_{\~{\@R_n}(Y)} Y  \cong X \times_Z Y$. Since both $\~{\@R_n}(X \times_Z Y) \rightarrow \~{\@R_n}(Y)$ and $\~{\@R_n}(X) \times_{\~{\@R_n}(Z)} \~{\@R_n}(Y) \rightarrow \~{\@R_n}(Y)$ satisfy this property, they must be equal. 
\end{proof}

\begin{proposition}
$((\~{\@R_n})_{!}, (\~{\@R_n})_*)$ is a Quillen adjunction.
\end{proposition}

\begin{proof}
By \cite{DHI}, Cor 8.3 it suffices to check that $\~{\@R_n}$ preserves étale covers and that it commutes with limits of finite diagrams of étale maps. This follows from Prop. \ref{preserving-etale} and Prop. \ref{preserving-fibre-products}.
\end{proof}

\begin{definition}
For any $n \geq 0$, the $n$-th Greenberg functor for $R$ is defined to be the right-derived functor $\|R (\~{\@R_n})_*: St(\Spec(R_n)) \rightarrow St(\Spec(R_0))$ and is denoted by $Gr_n$. If $X$ is a stack over $R$, we abuse notation and write $Gr_n(X)$ instead of $Gr^R_n(X \times_{\Spec(R)} \Spec(R_n))$.    
\end{definition}

\begin{convention}
Actually it would be appropriate to include a reference to $R$ in the notation for the $n$-th Greenberg functor. However, we will avoid this to prevent the notation from becoming too cumbersome. This will not lead to any confusion. 
\end{convention}

\begin{proposition}
Let $n \geq 0$ be an integer. The functor $Gr_n$ has the following properties:
\begin{itemize}
\item[(1)] $Gr_n$ preserves homotopy fibre products.
\item[(2)] $Gr_n$ takes schemes of finite type over $R_n$ to schemes of finite type over $k$. 
\item[(3)] $Gr_n$ takes smooth (étale, unramified) morphisms between schemes of finite type over $R_n$ to smooth (resp. étale, unramified) morphisms between schemes of finite type over $k$.
\item[(4)] $Gr_n$ preserves epimorphisms of stacks.
\item[(5)] $Gr_n$ takes sft-Artin stacks over $R_n$ to sft-Artin stacks over $k$. 
\end{itemize}
\end{proposition}

\begin{proof}
(1) is obvious since $(\~{\@R_n})_*$ is a right Quillen functor. (2) is proved in \cite{Gr}. 

(3) is stated in \cite{BLR} without proof. A proof is included here for completeness. Suppose $char(R) = 0$ and $char(k)=p \neq 0$. Let $x =(x_0,\ldots, x_{n-1})$ and $y=(y_0,\ldots,y_{n-1})$ denote generic elements of $W_n$. Then all the monomials appearing in the polynomials defining the product $xy$ involve both the $x_i$ and $y_i$ to non-zero degree. Indeed, this follows immediately by observing that $W(\bar{k})$ is an integral domain for any algebraic closure $\bar{k}$ of $k$. Thus we see that for any ring $A$, if $a = (a_0,\ldots,a_{n-1})$ and $(b_0, \ldots,b_{n-1})$ are elements of $W_n(A)$ such that the $a_i$ are in an ideal $I$ and the $b_i$ are in an ideal $J$, then the coordinates of $ab$ are in the ideal $IJ$. By examining the multiplication rule on $\@R_n$, we see that this argument continues to hold with $\@R_n$ in place of $W_n$. Also note that the set of elements having all coordinates in an ideal $I$ is precisely the kernel of $\@R_n(A) \rightarrow \@R_n(A/I)$. Thus we see that if $I$ is a nilpotent ideal in $A$, then the kernel of $\@R_(A) \rightarrow \@R_n(A/I)$ is a nilpotent ideal in $\@R(A)$. This proves that $Gr_n$ preserves the property of being formally smooth, étale or unramified. This proves (3) in the unequal characteristic case. The argument in the case $char(R) = char(k)$ is similar but simpler. 

To prove (4), it suffices to see that for any affine scheme $U$ over $k$, any étale cover of $\~{\@R_n}(U)$ can be refined by a cover of the form $\{\~{\@R_n}(U_i) \rightarrow \@R_n(U\}_i$. But this is obvious since $U$ is a closed subscheme of $\@R(U)$ defined by a nilpotent ideal.

(5) now follows immediately since the notion of an Artin stack is defined in terms of affine schemes, smoothness and homotopy fibre products.  
\end{proof}

Let $m \geq n$ be non-negative integers. Let $U$ be an affine scheme over $k$. The ring-scheme homomorphism $\@R_m \rightarrow \@R_n$ induces a morphism $e^m_n: \~{\@R_n}(U) \rightarrow \~{\@R_m}(U)$. These morphisms induce the ``truncation morphisms" as follows: 

\begin{definition}
Let $X$ be a stack over $R$. The truncation morphism $\tau^m_{n,X}: Gr_m(X) \rightarrow Gr_n(X)$ is the one that maps 
\[
x_m \in Gr_m(X)(U): \~{\@R_m}(U) \longrightarrow X
\]
to 
\[
x_n \in Gr_n(X)(U): \~{\@R_n}(U) \stackrel{e^m_n}{\longrightarrow} \~{\@R_m}(U) \stackrel{x_m}{\longrightarrow X} 
\]
for any affine scheme $U$ over $k$. 

Also, for every $n$ and for any affine scheme $U$ let $\tau_{n,X}(U)$ denote the function 
\[
\pi_0(X(\@R(U))) \rightarrow \pi_0(X(\@R_n(U))) \equiv \pi_0(Gr_n(X)(U)) \text{.}
\] 
If there is no risk of confusion, we will write $\tau^m_n$ instead of $\tau^m_{n,X}$. We will also write $\tau_{n}$ instead of $\tau_{n,X}(U)$ (i.e. we will omit the reference to both $U$ and $X$).  
\end{definition}

\section{Lifting $R$-valued points to an atlas}
\label{lifting-R-valued-points-to-atlas}

Let $X$ be an Artin stack and let $p: U \rightarrow X$ be a smooth atlas with $U$ being an affine scheme. Then if $K$ is a field and $x: \Spec(K) \rightarrow X$ is a morphism, we may not always be able to find a lift $u: \Spec(K) \rightarrow X$. However, given $x$, $p$ can be chosen appropriately so that a lift does exist, as we will prove in Lemma \ref{lifting-points-to-atlas}. If $X$ is an sft-Artin stack over a noetherian base scheme, $p$ can be chosen to be independent of $x$. 

Lemma \ref{lifting-points-to-atlas} is proved in (\cite{Kn}, Thm. II.6.4) for algebraic spaces and the argument is generalized(\cite{LMB}, Chapter 6) for Artin $1$-stacks. The proof for Artin $n$-stacks is based on the same technique with some small modifications. 

Let $X \rightarrow Y$ be a morphism of stacks. Then $(X/Y)^d$ denotes the $d$-fold fibre product 
\[
\underbrace{X \times^h_Y X \times^h_Y \ldots \times^h_Y X}_{\text{$d$ times}} \text{.}
\]
Let $\!S_d$ denote the symmetric group on $d$ letters. $\!S_d$ acts on $(X/Y)^d$ by permuting the factors and the quotient stack is denoted by $\Sigma_d(X/Y)$. We observe that this construction is well-behaved with respect to base changes of the form $Y^{\prime} \rightarrow Y$.

\begin{lemma}
\label{lifting-points-to-atlas}
Let $S$ be an arbitrary base scheme and let $X$ be an Artin stack over $S$. Let $K$ be a field and suppose we have a morphism $x: \Spec(K) \rightarrow X$. Then there exists a diagram
\[
\xymatrix{
 & V \ar[d]^{\phi} \\
\Spec(K) \ar[ur]^{v} \ar[r]_x & X
}
\]
which commutes upto homotopy, where $V$ is an affine scheme and $\phi$ is smooth. 
\end{lemma}

\begin{proof}
Choose a smooth morphism $U \rightarrow X$, where $U$ is an affine scheme, such that the stack $U_x:=U \times_{X,x}^h \Spec(K)$ is non-empty.  Since $U_x$ is non-empty, there exists a $K$-morphism $x^{\prime}:\Spec(K^{\prime}) \rightarrow U_x$ where $K^{\prime}$ is a separable finite field extension of $K$. Let $d = [K^{\prime}: K]$. Let $L$ be a Galois extension of $K$ containing $K^{\prime}$, so that $\Spec(K^{\prime}) \times_{\Spec(K)} \Spec(L) = \coprod_{J} \Spec(L)$ where $J$ is a set of representatives for the cosets of $Gal(L/K^{\prime})$ in $Gal(L/K)$. Note that the action of $Gal(L/K)$ on $\Spec(K^{\prime}) \times_{\Spec(K)} \Spec(L)$ merely permutes the components of $\coprod_{J} \Spec(L)$. 

Clearly, $x^{\prime}$ induces a morphism $\Spec(L) \times \{1,\ldots,d \} \rightarrow U_x \times_{\Spec(K)}^h \Spec(L)$, i.e. a morphism $\Spec(L) \rightarrow (U_x/\Spec(K))^d \times_{\Spec(K)}^{h} \Spec(L)$. On composing with the projection, we get a morphism $\Spec(L) \rightarrow (U_x/\Spec(K))^d$. On composing with the quotient map for the $\!S_d$ action, we get a morphism $\Spec(L) \rightarrow \Sigma_d(U_x/\Spec(K))$. As we noted above, the action of $Gal(L/K)$ permutes the components of $\Spec(K^{\prime}) \times_{\Spec(K)} \Spec(L) = \coprod_{J} \Spec(L)$ and thus the $L$-valued point $\Spec(L) \rightarrow \Sigma_d(T_1/T)$ that we have obtained is invariant under the action of $Gal(L/F)$, i.e. it is an $F$-valued point. This gives us an $F$-valued point of $\Sigma_d(U/X)$. 

Since $U\rightarrow X$ is smooth, so is $(U/X)^d \rightarrow X$. The quotient morphism $(U/X)^d \rightarrow \Sigma_d(U/X)$ is obviously a smooth covering map and thus $\Sigma_d(U/X) \rightarrow X$ is smooth. If $X$ is an $n$-stack, it is clear that $(U/X)^d$ is an $(n-1)$-stack. Let $s: Z \rightarrow (U/X)^d$ be a geometric point (i.e. $Z$ is the spectrum of a separably closed field) of $(U/X)^d$ and let $t$ be its image in $\Sigma_d(U/X)$. Then we have the long exact sequence of homotopy groups 
\begin{eqnarray*}
\ldots \rightarrow \pi_i(F_t(Z),s) \rightarrow \pi_i((U/X)^d(Z), s) \rightarrow \pi_i(\Sigma_d(U/X))(Z),t) \\
\rightarrow \pi_{i-1}(F_t(Z),s) \rightarrow \ldots
\end{eqnarray*}
where $F_t$ is the fibre of $(U/X)^d \rightarrow \Sigma_d(U/X)$ at $t$. Since $F_t$ is isomorphic to $Z \times \!S_d$, we see immediately that the $\pi_i((U/X)^d(Z),s) \cong \pi_i(\Sigma_d(U/X)(Z),t)$ for $i>1$. Thus if $X$ is an $n$-stack for $n \geq 2$, $\Sigma_d(U/X)$ is an $(n-1)$-stack. Now replace $X$ by $\Sigma_d(U/X)$ and repeat the procedure until we come to the case $n=1$. If $X$ is an Artin $1$-stack, $(U/X)^d$ is an algebraic space and thus $\Sigma_d(U/X)$ is the quotient of an algebraic space under the action of a finite group. 

Now the required result follows immediately from (\cite{LMB} Thm. 6.1) and we briefly reproduce the argument. Choose the usual embedding of $\!S_d$ into the group scheme $GL_{n,S}$. Let $V^{\prime}$ be the quotient of the action of $\!S_d$ on $(U/X)^d \times_S GL_{n,S}$. Then it can be checked that $V^{\prime}$ is an algebraic space and that $V^{\prime}$ is a $GL_{n,S}$-torsor over $\Sigma_{d}(U/X)$. Thus for any morphism $T \rightarrow \Sigma_d(U/X)$ from a semi-local scheme into $\Sigma_d(U/X)$, there exists a lift $T \rightarrow V^{\prime}$. In particular, there exists a morphism $v^{\prime}: \Spec(K) \rightarrow V^{\prime}$ lifting $x$. 

Finally, now we apply (\cite{Kn}, Thm. II.6.4) to construct an étale map $V \rightarrow V^{\prime}$ such that $V$ is an affine scheme and such that there exists a $v: \Spec(K) \rightarrow V$ lifting $v^{\prime}$. 
\end{proof}

We note that in the above lemma, the morphism $V \rightarrow X$ is specifically chosen for the given morphism $x: \Spec(K) \rightarrow X$. However, for sft-Artin stacks over a noetherian base, we are able to strengthen this result.

\begin{definition}
Let $p: X \rightarrow Y$ be a morphism of Artin stacks. 
\begin{itemize}
\item[(1)] We say that $p$ is of f-class $n$ if for any field $K$ and any morphism $y: \Spec(K) \rightarrow Y$, there exists a finite field extension $L$ of $K$ with $[L:K] \leq n$ such that there exists a morphism $x: \Spec(L) \rightarrow X$ such that the square
\[
\xymatrix{
\Spec(L) \ar[r]^x \ar[d] &  X \ar[d]^f \\
\Spec(K) \ar[r]^y  & Y 
}
\]
commutes upto homotopy.
\item[(2)] We say that $p$ is f-surjective if it is of f-class $0$. In other words, for any field $K$, the morphism $\pi_0(X(K)) \rightarrow \pi_0(Y(K))$ is surjective. 
\end{itemize}
\end{definition}

\begin{lemma}
\label{finite-f-class}
Let $p: X \rightarrow Y$ be a morphism of sft-Artin stacks over a noetherian base scheme $S$. Then there exists a integer $n$ such that $p$ is of f-class $n$. 
\end{lemma}

\begin{proof}
We break the proof down into cases:

\noindent
\textit{(Case 1) $X$ and $Y$ are affine schemes:} Then if $Y = \Spec(R)$ then $X = \Spec(R[X_1, \ldots, X_r]/I)$ where $I=<h_1, \ldots, h_s>$ is some finitely generated ideal of the noetherian ring $R[X_1, \ldots, X_r]$. Now if $y$ is a point of $Y$ given by a ring homomorphism $R \rightarrow K$ for some field $K$, the fibre $X \times_{Y,y} \Spec(K)$of $p$ over $Y$ is a scheme of finite type over $K$ whose underlying set is the set of solutions of the images of the polynomials $h_i$ in the ring $K[X_1, \ldots, X_r]$. Clearly, there exists a number $n$ depending only on the degrees of the polynomials $h_i$ such that there exists a field $L$ with $[L:K]$ and a $K$-morphism $\Spec(L) \rightarrow X \times_{Y,y} \Spec(K)$. This proves the result when $X$ and $Y$ are affine schemes. 

\noindent
\textit{(Case 2) $X$ and $Y$ are algebraic spaces:} All algebraic spaces of finite type over $S$ have a finite stratification into locally closed subpsaces which are isomorphic to affine schemes. Now applying (Case 1) to suitably chosen stratifications of $X$ and $Y$ we get the desired result.

\noindent
\textit{(Case 3) $Y$ is an affine scheme $p$ is arbitrary:} Note that if we have morphisms $F \rightarrow G \rightarrow H$ of Artin stacks over $S$ and if the statement of the lemma is true for the morphisms $F \rightarrow H$ then it is also true for the morphism $G \rightarrow H$. Now let $U \rightarrow X$ be a smooth atlas of $X$ with $U$ being an affine scheme of finite type over $k$. The result follows from (Case 1). 

\noindent
\textit{(Case 4) The general case:} Let $k \geq -1$ be an integer such that $Y$ is $k$-geometric and $p$ is $k$-representable (see \cite{To} for this terminology). We prove the result by induction on $k$. The case $k=-1$ is covered in (Case 1). Suppose the result is true for $k \leq m-1$. 

Now suppose $Y$ is $m$-geometric and $Y$ is $m$-representable. Let $V \rightarrow Y$ be a smooth atlas where $V$ is an affine scheme. By the observation in (Case 3), it is enough to prove the result for the composition of the morphisms $V \times_Y X \rightarrow X \rightarrow Y$. But this is also the composition of the morphisms $V \times_Y X \rightarrow V \rightarrow Y$. The statement of the lemmma holds for $V \times_Y X  \rightarrow V$ by (Case 3). The morphism $V \rightarrow Y$ is $(m-1)$-representable and thus the statement of the lemma holds for this morphism by the induction hypothesis. This completes the proof.
\end{proof}

\begin{lemma}
\label{f-surjective-atlas}
Let $X$ be an sft-Artin stack over a noetherian base scheme $S$. Then there exists an affine scheme $V$ and a smooth covering map $V \rightarrow X$ which is f-surjective. 
\end{lemma}

\begin{proof}
The proof is based on the arguments in the proof of Lemma \ref{lifting-points-to-atlas}. Thus we will refer to that proof for the details. 

Suppose $X$ is an $n$-stack. Choose any smooth covering map $U \rightarrow X$. By Lemma \ref{finite-f-class}, there exists an $n$ such that $U \rightarrow X$ is of f-class $n$. Then by the argument in the proof of Lemma \ref{lifting-points-to-atlas}, the map $\Sigma_{n!}(U/X)$ is f-surjective. Then $\Sigma_{n!}(U/X)$ is an $(n-1)$-stack. Proceeding as in the proof of Lemma \ref{lifting-points-to-atlas}, we get a smooth covering map $X^{\prime} \rightarrow X$ which is $f$-surjective and $X^{\prime}$ is an algebraic space. Thus it remains to construct an f-surjective smooth covering map $V \rightarrow X^{\prime}$ where $V$ is an affine scheme. 

Let $x$ be any point of $X^{\prime}$ and let $K$ be the residue field of $X^{\prime}$ at $x$. Let $\-{\{x\}}$ denote the closure of $x$ in $X^{\prime}$ and let $Z \subset \-{\{x\}}$ be an open dense subscheme of $\-{\{x\}}$. Using Lemma \ref{lifting-points-to-atlas} (or, more honestly, \cite{Kn}, Thm. II.6.4), there exists a smooth (or even étale) morphism $V_x \rightarrow X^{\prime}$ such that $x: \Spec(K) \rightarrow X$ can be lifted to $V_x$. Then it follows that there is a dense open subscheme $U_x \subset Z$ such that the immersion $U_x \rightarrow X^{\prime}$ has a lift $U_x \rightarrow V_x$. In particular $V_x \times_{X^{\prime}} U_x \rightarrow U_x$ is f-surjective. 

Now we can apply this construction to all the generic points of the top-dimensional components of $X^{\prime}$. This gives us a dense open subscheme $U_0 \subset X^{\prime}$ and a smooth map $V_0 \rightarrow X^{\prime}$, the image of which contains $U_0$ and such that $V_0 \times_{X^{\prime}} U_0 \rightarrow U_0$ is f-surjective. Then we apply this argument to all the generic points of the top-dimensional components of $X^{\prime} \backslash U_0$. Proceeding in this manner and using the fact that $X^{\prime}$ is a noetherian space, we get the required result. 
\end{proof}

\begin{corollary}
\label{visible-atlas}
Let $X$ be an sft-Artin stack over a complete discrete valuation ring $A$. Let $\alpha$ be a uniformizing parameter in $A$. Then $X$ has an atlas $U \rightarrow X$ where $U$ is an affine scheme of finite type over $A$ such that the maps $U(A) \rightarrow \pi_0(X(A))$ and $U(A/\alpha^{n+1}) \rightarrow \pi_0(X(A/\alpha^{n+1}))$ for all $n \geq 0$ are surjective. 
\end{corollary}

\begin{proof}
Let $K$ denote the residue field of $A$. Using Lemma \ref{f-surjective-atlas}, there exists an f-surjective smooth atlas $U \rightarrow X$ such that $U$ is an affine scheme. 

We prove that $U(A) \rightarrow \pi_0(X(A))$ is surjective. (The proof for the maps $U(A/\alpha^{n+1}) \rightarrow \pi_0(X(A/\alpha^{n+1}))$ is exactly the same.) Indeed, pick any morphism $t: \Spec(A) \rightarrow X$. It suffices to show that the smooth stack $U_t: U \times_{X,t}^h \Spec(A)$ has an $A$-valued point. By construction, $U_t$ has a $K$-valued point $u$. We construct an atlas $V_t \rightarrow U_t$ such that $V_t$ is a (smooth) affine scheme over $A$ and such that $u$ lifts to $V_t$. Now we already know by Hensel's lemma that this lift can be extended to an $A$-valued point of $V_t$ which gives an $R$-valued point of $U_t$. This completes the proof.
\end{proof}

\begin{remark}
If, in Lemma \ref{visible-atlas}, $K$ is a finite field, the result follows from Lemma \ref{lifting-points-to-atlas} itself since then $\pi_0(X(K))$ is known to be a finite set by (\cite{To}, Prop. 3.5). On the other hand, once we have Lemma \ref{f-surjective-atlas}, we get an alternative proof of the fact that $\pi_0(X(K))$ is finite. Indeed, choose a smooth atlas $U \rightarrow X$ where $U$ is an affine scheme. Then by Lemma \ref{finite-f-class}, it follows that there exists a finite algebraic extension $L$ of $K$ such that any $K$-valued point of $X$ lifts to an $L$-valued point of $U$. But the number of $L$-valued points of $U$ are known to be finite. 
\end{remark}

\section{$p$-adic measure on Artin stacks}
\label{p-adic-measure-on-stacks}

We recall the notation from section 1 that $R$ is a complete discrete valuation ring with a finite residue field $k$ of cardinality $q = p^r$, $\omega$ is a uniformizing parameter in $R$ and $R_n:= R/< \omega^{n+1}>$ for each $n \geq 0$. In this section, we examine the numbers $\#X(R_n)$ for $n \geq 0$. We also define a $p$-adic measure on $\pi_0(X(R))$. For this, we view $\pi_0(X(R))$ as a locally compact topological space by the quotient topology given by the map $U(R) \rightarrow X(R)$ where $U \rightarrow X$ is an f-surjective smooth atlas with $U$ being an affine scheme. It is easily seen that this topology is independent of the choice of the atlas $U \rightarrow R$ since any two such atlases $U_1 \rightarrow X$ and $U_2 \rightarrow X$ have a common refinement (for example, choose a f-surjective smooth covering map $U_3 \rightarrow U_1 \times^h_X U_2$). The $p$-adic measure will be a Borel measure on this space. 

For an stf-Artin stack $X$ over $k$, we note that the counting formula (\ref{toen-counting}) defines a measure $\pi_0(X(k))$. For any subset $A \subset \pi_0(X(k))$, we denote this measure by $\#A$. To be precise, we write
\begin{equation}
\label{counting-measure}
\#A:= \sum_{x \in A} \prod_{i>0} |\pi_i(X(k),x)|^{{(-1)}^i} \text{.}
\end{equation}

\begin{lemma}
\label{counting-by-fibres}
Let $p: F \rightarrow G$ be a morphism of sft-Artin stacks over $k$. Let $y \in \pi_0(G(k))$. Let $F_y = F \times_G^h \Spec(k)$. Then 
\[
\#p^{-1}(y) = (\#F_y(k)) \cdot (\#\{y\}) 
\]
where $p^{-1}(y) = \{ x \in \pi_0(F(k))| p(x)=y\}$. 
\end{lemma}
\begin{proof}
Let $i_y: F_y \rightarrow F$ be projection morphism. the Let $x \in p^{-1}(y)$. Let $x^{\prime} \in F(k)$ such that $i_y \circ x^{\prime} \cong x$. Then by the long exact sequence of homotopy groups corresponding to the fibration sequence $F_y(k) \rightarrow F(k) \rightarrow G(k)$, we have 
\[
|i_y^{-1}(x)| \cdot \prod_{i=1}^{\infty} \left( \frac{|\pi_i(F_y(k),x^{\prime})| \cdot |\pi_i(G(k),y)|}{|\pi_i(F(k),x)|} \right)^{(-1)^i}  = 1 \text{.}
\]
Thus $(\#\{x^{\prime}\}) \cdot (\#\{y\}) = |i_y^{-1}(x)|^{-1} \cdot (\#\{x\})$. Summing up over all $x^{\prime} \in i_y^{-1}(x)$, we get
\[
(\# i_y^{-1}(x)) \cdot (\#\{y\}) = |i_y^{-1}(x)|^{-1} \cdot (\#\{x\}) \text{.}
\]
Summing up over all $x \in p^{-1}(y)$, we get 
\[
(\# F_y(k)) \cdot (\#\{y\}) = \#p^{-1}(y) \text{.}
\]
\end{proof}

\begin{proposition}
\label{counting-on-smooth-stacks}
Let $X$ be a smooth sft-Artin stack over $R$ with $\dim(X/R)=d$. Let $n \geq 0$ be an integer and let $x \in \pi_0(Gr_n(X)(k))$. Then $\#(\tau_n^{n+1})^{-1}(x) = q^d \cdot \#\{x\}$. 
\end{proposition}

\begin{proof}
Suppose $X$ be an sft-Artin stack which is $m$-geometric. We prove the result by induction on $n$. When $m = -1$, i.e. when $X$ is an affine scheme, the result just a consequence of Hensel's lemma. Suppose the result has been proved for $n \leq m$. We prove the induction hypothesis for $n = m$. 

Let $f: U \rightarrow X$ be a f-surjective smooth atlas such that $U$ is an affine scheme with $\dim(U/R)=e$. Let $x^{\prime} \in (\tau_n^{n+1})^{-1}(x)$ and let $\~x$ be an element of $(\tau_{n+1})^{-1}(x^{\prime})$ (it is easy to see that $\~x$ exists because $X$ is smooth). Let $F = U \times^h_{X,\~x} \Spec(R)$. Then $F$ is a smooth sft-Artin stack over $R$ which is $(m-1)$-geometric. 

We have 
\[
|Gr_n(f)^{-1}(x)| = (\#\pi_0(Gr_n(F)(k)))\cdot(\#\{x\}) 
\]
and, similarly,
\[
|Gr_{n+1}(f)^{-1}(x^{\prime})| = (\#\pi_0(Gr_{n+1}(F)(k)))\cdot(\#\{x^{\prime}\}) \text{.}
\]
Then by the induction hypothesis, we have \[
\#\pi_0(Gr_{n+1}(F)(k)) = q^{(e-d)} \cdot \#\pi_0(Gr_n(F)(k)) \text{.}
\]
Thus we have 
\[
|Gr_{n+1}(f)^{-1}(x^{\prime})|  = q^{(e-d)}\cdot |Gr_n(p)^{-1}(x)| \frac{\#\{x^{\prime}\}}{\#\{x\}} \text{.} 
\]
Letting $x^{\prime}$ vary over the set $(\tau_n^{n+1})^{-1}(x)$ and summing up, we get 
\[
|Gr_{n+1}(f)^{-1}((\tau_n^{n+1})^{-1}(x))| = q^{(e-d)}\cdot |Gr_n(f)^{-1}(x)| \frac{\#(\tau_n^{n+1})^{-1}(x)}{\#\{x\}} \text{.}
\]
But 
\[
Gr_{n+1}(f)^{-1}((\tau_n^{n+1})^{-1}(x)) = (\tau_n^{n+1})^{-1}(Gr_n(f)^{-1}(x)) \text{.}
\]
Thus, applying the induction hypothesis for $n=-1$, we have 
\[
|Gr_{n+1}(f)^{-1}((\tau_n^{n+1})^{-1}(x))|  = q^e \cdot |Gr_n(f)^{-1}(x)|
\]
which completes the proof of the induction hypothesis for $n=m$. 
\end{proof}

Now let $X$ be an arbitrary sft-Artin stack over $R$ with $\dim(X/R)=d$. Let $f: U \rightarrow X$ be an f-surjective smooth atlas where $U$ is an affine scheme with $\dim(U/R) = e$. For any $s \in \pi_0(X(k))$, let $\-U_s := U \times^h_{X,s} \Spec(k)$ and let $m_s:= \# U_s(k)$. Note that $m_s \neq 0$ for all $s$. 

Let $A$ be any subset of $\pi_0(Gr_n(X)(k))$. Let $A_s$ denote the set $A \cap (\tau^n_0)^{-1}(s)$ so that $A = \coprod_{s \in S} A_s$. By Lemma \ref{counting-by-fibres} and Prop. \ref{counting-on-smooth-stacks}, we have
\[
|Gr_n(f)^{-1}(A_s)| = q^{n(e-d)} \cdot m_s \cdot \#A_s \text{.}
\]
Thus 
\begin{eqnarray*}
q^{-nd} \cdot \#A & = & q^{-nd} \cdot \left( \sum_{s \in \pi_0(X(k))} \#A_s \right) \\
                            & = & q^{-nd} \cdot \left( \sum_{s \in \pi_0(X(k))} \frac{|Gr_n(f)^{-1}(A_s)|}{m_s \cdot q^{-n(e-d)}}\right) \\
                            & = & \sum_{s \in \pi_0(X(k))} \left(\frac{1}{m_s} \right) \cdot \left(\frac{|Gr_n(f)^{-1}(A_s)|}{q^{-ne}} \right)
\end{eqnarray*}
In particular, if $A \subset \pi_0(X(R))$, then
\begin{equation}
\label{p-adic-measure-stacks-equation}
q^{-nd} \cdot \#(\tau_n(A)) = \sum_{s \in \pi_0(X(k))} \left(\frac{1}{m_s} \right) \cdot \left(\frac{|Gr_n(f)^{-1}(\tau_n(A)_s)|}{q^{-ne}} \right)
\end{equation}
This leads us to define $p$-adic measure as follows:
\begin{definition} 
\label{p-adic-measure-definition}
With the above notation, note that a set $A \subset \pi_0(X(R))$ is a Borel subset if and only if $f^{-1}(A) \subset U(R)$ is a Borel subset. We define the $p$-adic measure of such a set by
\[
\mu^f_d(A):= \sum_{s \ in S} \left(\frac{1}{m_s} \right) \cdot \mu_e(f^{-1}(A_s)) \text{.}
\]
\end{definition}

\begin{lemma}
With the above notation, $\mu^f_d$ is independent of the choice of $f$. 
\end{lemma}

\begin{proof}
This follows almost from the definition. Indeed, the limit of the left-hand side of equation \ref{p-adic-measure-on-stacks} as $n \rightarrow \infty$, if it exists, clearly does not depend on $f$. This limit exists if the limit of the right-hand side of equation (\ref{p-adic-measure-stacks-equation}) exists as $n \rightarrow \infty$. This is clearly so if $A$ is an open subset (for the topology described at the beginning of this section). If $A$ is a Borel subset, so is $f^{-1}(A)$. Since Borel subsets are outer regular for the $p$-adic measure on $U(R)$, we see that $\mu^f_d$ is independent of $f$ on an arbitary Borel subset of $\pi_0(X(R))$.  
\end{proof}

The arguments above have given us the following result which we restate explicitly. 

\begin{proposition}
Let $X$ be an sft-Artin stack over $R$ with $\dim(X/R)=d$. Then the sequences $\{p^{-nd} \#\pi_0(X(R_n))\}_{n=0}^{\infty}$ and $\{p^{-nd} \#\tau_n(\pi_0(X(R)))\}_{n=0}^{\infty}$ both converge to $\mu_d(\pi_0(X(R)))$. 
\end{proposition}

Finally we look at the power series $\~P_X(T)$ and $P_X(T)$.

\begin{proposition}
Let $X$ be a sft-Artin stack over $R$. Then the power series $P_X(T)$ and $\~P_X(T)$ are rational functions of $T$. 
\end{proposition}
\begin{proof}
With the notation as above, writing $A_n:= \pi_0(X(R_n)) = \pi_0(Gr_n(X)(k))$, we have the equalities
\begin{eqnarray*}
\~P_X(T) & = & \sum_{s \in \pi_0(X(k))} \sum_{n=0}^{\infty} \#(A_n)_s T^n \\
       & = & \sum_{s \in \pi_0(X(k))} \frac{1}{m_s \cdot p^{n(e-d)}}\sum_{n=0}^{\infty} |Gr_n(f)^{-1}((A_n)_s)| T^n \text{.}
\end{eqnarray*}
Now it follows from (\cite{De}, Thm. 4.1) that $\~P_X(T)$ is a rational function of $T$. The proof for $P_X(T)$ is similar. 
\end{proof}

While the above argument is adequate to establish the rationality of $P_X(T)$, in order to prove a ``uniform rationality" theorem, it is useful to view the power series $P_X(T)$ a little differently. We recall the definition of the singular locus of a stack. 

\begin{definition}
\label{singular-locus}
Let $X$ be an sft-Artin stack over an affine scheme $S$. Then the singular locus $X_{sing}$ is a closed substack of $X$ defined as follows: 
\begin{itemize}
\item[(1)] If $X$ is an affine scheme of dimension $d$ over $S$, then $X_{sing}$ is the closed subscheme of $X$ defined by the $d$-th Fitting ideal of $\Omega_{X/S}$ (the module of relative differentials of $X$ over $S$. 
\item[(2)] In general, let $f: U \rightarrow X$ be a smooth atlas with $U$ being an affine scheme of finite type over $S$. Then $X_{sing}$ is the closed substack of $X$ which is the image of $U_{sing} \rightarrow X$. 
\end{itemize}
\end{definition}
In (2) above, it is easy to check that the image stack of $U_{sing} \rightarrow X$ is a closed substack of $X$ and that the definition of $X_{sing}$ does not depend on the choice of the atlas $f$. 

\begin{definition}
\label{Q_X-definition}
Let $X$ be an sft-Artin stack over $R$ and let $X_{sing}$ denote its singular locus over $R$. Then the power series $Q_X(T)$ is defined as 
\[
Q_X(T):= P_X(T) - P_{X_{sing}}(T) \text{.}
\]
\end{definition}

The reason why the power series $Q_X(T)$ is useful is that the problem of proving the rationality of $P_X(T)$ is equivalent to that of proving the rationality of $Q_X(T)$. Indeed, suppose we know that the power series $Q_X(T)$ is a rational function for any sft-Artin stack over $R$. Then we can prove the rationality of $P_X(T)$ by noetherian induction on the closed substacks of $X$. Indeed, if $P_{X_{sing}}(T)$ and $Q_X(T)$ are both known to be rational functions, it follows that $P_X(T)$ is a rational function. The advantage here is that the coefficients of $Q_X(T)$ have a simple description in terms of $p$-adic measure. 

\begin{lemma}
\label{Q_X}
Let $X$ be an sft-Artin stack over $R$ with $\dim(X/R) = d$. Then $Q_X(T) = \sum_{n=0}^{\infty} q^{nd}\mu_d(M_n)T^n$ where $M_n$ is the subset of $\pi_0(X(R))$ given by
\[
M_n = \pi_0(X(R)) \backslash \tau_{n,X}^{-1}(X_{sing}(R_n))
\]
\end{lemma}

\begin{proof}
Let $S_n:= \tau_n(\pi_0(X(R))) \backslash \tau_n(\pi_0(X_{sing}(R))) \subset \pi_0(X(R_n))$ so that the coefficient of $T^n$ in $Q_X(T)$ is $\#S_n$. We wish to prove $\mu_d(M_n) = \#S_n/q^{nd}$. Note that $M_n = \tau_n^{-1}(S_n)$. 

First suppose that $X$ is an affine scheme. Then it is known that for any $n$, $x \in S_n$ implies that $\#[(\tau_n^{n+1})^{-1}(x) \cap \tau_n(\pi_0(X(R)))]=q^d$ (for example, see the argument in \cite{Lo}, Lemma 9.1).  

The case of a general sft-Artin stack follows from equation (\ref{p-adic-measure-stacks-equation}) and our definition of $p$-adic measure. 
\end{proof}

\section{Motivic measure on stacks}
\label{motivic-measure-on-stacks}

Let $k$ be a field of characteristic zero and let $\Field_k$ denote the category of field extensions of $k$. We will now associate a ``motivic measure" to definable subassignments on Artin stacks over $k[[t]]$. In doing so, we will assume familiarity with the theory of motivic integration as presented in \cite{CL}.  For the sake of completeness, we recall some of the definitions and notation from that work. We are essentially reproducing the summary from (\cite{CL2}, Section 2) while incorporating the changes that are necessary for our setting. 

For any field extension $K$ of $k$, we consider the power series ring $K[[t]]$. This is a discrete valuation ring with valuation $\ord: K[[t]]\backslash \{0\} \rightarrow \|Z$. Let $\ac: K[[t]] \rightarrow K$ be the ``angular component" map, i.e. 
\[
\ac(x)= \begin{cases}
xt^{-\ord(x)} \mod{t} & \text{if $x \neq 0$} \\
0 & \text{if $x=0$.}
\end{cases}
\]
In order to work with these power series rings, we use the language of Denef-Pas, which we denote by $\Ldp$. This is a $3$-sorted language
\[
\Ldp:= (\Lval, \Lres, \Lord, \ord, \ac )
\]
with three sorts, $\Val$, $\Res$ and $\Ord$ corresponding to the valuation ring, the residue field and the order group. $\Lval$ and $\Lres$ are equal to the language of rings $\{+, -, \cdot, 0, 1 \}$ while the language $\Lord$ is the Presburger language 
\[
\{ +, -, 0, 1, \leq \} \cup \{  \equiv_n | m \in \|N, n > 1\}
\]
wjere $\equiv_n$ is interpreted as congruence modulo $n$.  

Recall that if $\@C$ is a category $\@C$ and $F: \@C \rightarrow \Sets$ is a functor, a \textit{subassignment} $h$ of $F$ is a rule which assigns a subset $h(C) \subset F(C)$ for each object $C$ of $\@C$. For subassignments, the usual set-theoretic notions such as $\cup$, $\cap$, $\subset$, etc. are defined objectwise. In particular the subassignments of a fixed functor form a Boolean algebra. 

We use this notion with $\@C = \Field_k$. Consider a triple $(\@X, X, r)$ where $\@X$ is an Artin stack over $k[[t]]$, $X$ is an Artin stack over $k$ and $r \geq 0$ is an integer. Consider the functor be the functor
\[
(\@X \times X \times \|Z^r)(K):= \pi_0(\@X(K[[t]]) \times \pi_0(X(K)) \times \|Z^r  \text{.}
\]
When any of the elements in the triple $(\@X, X, r)$ are trivial (i.e. $\@X = \Spec(k[[t]])$, $X = \Spec(k)$ or $r=0$) we may abuse notation and simply omit to write them if there is no risk of confusion. When $\@X = \|A^n_{k[[t]]}$ and $X = \|A^m_k$, the above functor is denoted by $h[m,n,r]$. 

\begin{remark}
\label{remark-about-subassignments}
Note that we are considering $K[[t]]$-valued points when we define a subassignment while in (\cite{CL}, \cite{CL2}) one considers $K((t))$-valued points. However, when one is working with a separated scheme $X$, $X(K[[t]])$ maps injectively into $X(K((t)))$ and thus we may apply the results regarding motivic measure on such schemes without any problems. We will use these results only for affine schemes.  
\end{remark}

We will now define what it means for a subassignment of such a functor to be definable in the language of Denef-Pas. Given such a triple $(\@X,X,r)$, we choose f-surjective smooth atlases $\@U \rightarrow \@X$ and $U \rightarrow X$ such that $\@U$ and $U$ are affine schemes over $k[[t]]$ and $k$ respectively. Then consider the triple $(\@U, U, r)$. There is an obvious morphism of functors $f: \@U \times U \times \|Z^r \rightarrow \@X \times X \times \|Z^r$. We say that a subassignment of $\@X \times X \times \|Z^r$ is definable if and only if its preimage in $\@U \times U \times \|Z^r$ is definable by a formula in the language $\Ldp$ with $\Val$ coefficients in $k[[t]]$ and $\Res$ coefficients in $k$. It is easy to see that this notion is independent of the choice of $\@U \rightarrow \@X$ and $U \rightarrow X$. 

Suppose  $(\@X, X, r)$ and $(\@Y, Y, s)$ are two triples with $\@X$ and $\@Y$ (resp. $X$ and $Y$) being sft-Artin stacks over $k[[t]]$ (resp. $k$). Suppose we are given morphisms of stacks $\@X \rightarrow \@Y$ and $X  \rightarrow Y$ along with a linear map $\|Z^r \rightarrow \|Z^s$. Then this induces a morphism of the functors $\@X \times X \times \|Z^r \rightarrow \@Y \times Y \times \|Z^s$. We call such a morphism a \textit{geometric morphism}. 

Now consider the category $\@D$ whose objects are pairs $(S, (\@X, X, r))$ where $(\@X, X, r)$ is a triple as above and $S$ is a definable subassignement of $\@X \times X \times r$. A morphism $(S, (\@X,X,r)) \rightarrow (T, (\@Y, Y, s))$ is a geometric morphism $\@X \times X \times \|Z^r \rightarrow \@Y \times Y \times \|Z^s$ which maps $S$ into $T$. We say that morphism is an geometric equivalence if it induces a weak equivalence $S(K) \rightarrow T(K)$ for every field extension $K$ of $k$ and let us denote the class of geometric equivalences by $\@W$. (Here we view $S(K) \subset \pi_0(\@X(K[[t]]) \times \pi_0(X(K)) \times \|Z^r$ not just as a ordinary set, but as a set whose objects are homotopy types.)  Now consider the category $\@W^{-1}\@D$ obtained from $\@D$ by localizing with respect to $\@W$.  It is easy to see that $\@W$ is a left-multiplicative system of morphisms thus the localization makes sense. Indeed, any morphism $S \rightarrow T$ in the localization is given by an equivalence $W \rightarrow S$ and a geometric morphism $W \rightarrow T$. We refer to the morphisms in $\@W^{-1}\@D$ as \textit{definable morphisms}. (Note that unlike the case in which we are working with definable subassignments of varieties, we cannot simply say that a morphism $f: S \rightarrow T$  is  definable if its graph $\Gamma_f$ is a definable subassignment of $S \times T$ since, in general, $\Gamma_f \rightarrow S \times T$ is not a monomorphism.)

In the above construction, if we restrict ourselves to objects of the type $(S, (\|A^n_{k[[t]]}, \|A^m_k, r))$, where where $n,m \geq 0$, we denote the resulting category by $\Def_k$. (This is exactly the category $\Def_k$ defined in \cite{CL} except for the slight difference explained in Remark \ref{remark-about-subassignments}.) More generally, for any subassignment $S$, we denote by $\Def_S$ the category of subassignments contained in $S \times \|A^n_{k[[t]]}\times \|A^m_k \|Z^r$ where $n,m \geq 0$. We denote by $\RDef_S$ the subcategory of $\Def_S$ consisting of subassignments of $S \times \|A^m_k$ for $m \geq 0$. We denote by $K_0(RDef_S)$ the corresponding Grothendieck ring (\cite{CL}, Section 5).

Let $A$ denote the ring $\|Z[\|L, \|L^{-1}, \{(1-\|L^{-o})^{-1}\}_{i>0}]$. Let $\@P(S)$ denote the ring of functions from the set of points of $S$ into $A$ generated by constant functions, definable functions from $S$ into $\|Z$ and functions of the form $\|L^{\beta}$ with $\beta: S \rightarrow \|Z$ being a definable morphism. We denote by $\@P^0(S)$ the subring of $\@P(S)$ generated by the characteristic functions of definable subassignments contained in $S$ and the constant function $\|L$. There is a natural ring homomorphism $\@P^0(S) \rightarrow K_0(RDef_S)$ sending $\|L$ to the class of $S \times \|A^1_k$ and sending the characterstic function $\mathbbm{1}_T$ of a subassignment $T \subset S$ to the class of $T$ itself (viewed as an element of $RDef_S$). Then the ring of constructible functions on $S$ is defined by
\[
\@C(S):= K_0(RDef_S) \otimes_{\@P^0(S)} \@P(S) \text{.}
\]

Let $\@X$ be a sft-Artin stack over $k[[t]]$ and let $S$ be a definable subassignment contained in $\@X$ (viewed as a definable subassignment itself). Then the Zariski closure $W$ of $S$ in $\@X$ is the intersection of all closed substacks $\@Y \subset \@X$ such that $S \subset \@Y$. We set $\dim(S) = \dim(\@Y/\Spec(k[[t]])$. More generally, if $S$ is a definable subassignment of a functor of the type $\@X \times X \times \|Z^r$ where $\@X$ and $X$ are sft-Artin stacks over $k[[t]]$ and $k$-respectively, then we define the dimension of $S$ to be the dimension of its projection to $\@X$. For every integer $d$, we denote by $\@C^{\leq d}(S)$ the ideal of $\@C(S)$ generated by the characteristic functions of subassigments $Z \subset S$ with $\dim(Z) \leq d$. This defines a filtration of $\@C(S)$ and we denote the associated graded group by $C(S) := \oplus_d C^d(S)$ (the group of constructible motivic Functions - note the capital `F') where $C^d(S) := \@C^{\leq d}(S)/\@C^{\leq d-1}(S)$. Note that $d$ can be negative but that the set of $d$ such that $\@C^{\leq d}(S) \neq \{0\}$ is bounded below.

Now suppose $S$ is in $\Def_k$ and $Z$ is in $\Def_S$. Then one can define (\cite{CL}, Section 10) a subgroup $I_SC(Z)$ of $C(Z)$ together with pushforward morphisms 
\[
f_!: I_SC(Z) \rightarrow C(Y)
\]
for every morphism $f: Z \rightarrow Y$ in $\Def_S$. When $S$ is simply the final object of $\Def_k$, and $f$ is the map $Z \rightarrow S$, then we write $f_!$ as $\mu$ and call it the motivic measure. 

More generally, if $\Lambda$ is an object of $\Def_k$, the above construction can be performed relative to $\Lambda$. Using relative dimension instead of dimension, we can obtain relative analogues $C(Z \rightarrow \Lambda)$ for $Z \rightarrow \Lambda$ in $\Def_{\Lambda}$. In particular, we obtain a morphism
\[
\mu_{\Lambda}: I_{\Lambda} C(Z \rightarrow \Lambda) \rightarrow \@C(\Lambda)= I_{\Lambda}C(\Lambda \rightarrow \Lambda) \text{.}
\]

Finally, we recall that for $f: Z \rightarrow \Lambda$, $f_!$ and $\mu_{\Lambda}$ are completely different notions. However, they do coincide when $\Lambda$ is a subassignment of $\|A^m_k \times \|Z^r$ (see \cite{CL}, Remark 14.2.3). (We will make use this in the following lemma.)

This completes our review of the basic terminology and results. Now, in the following two lemmas, we will try to adapt the essence of equation (\ref{p-adic-measure-stacks-equation}) to motivic integration. For an sft-Artin stack $\@X$ over $k[[t]]$, we will denote by $\-{\@X}$ the Artin stack $\@X \times_{\Spec(k[[t]])} \Spec(k)$ over $k$. If $p: \@X \rightarrow \@Y$ is a morphism of sft-Artin stacks over $k[[t]]$, let $\-p$ denote the pullback $\-{\@X} \rightarrow \-{\@Y}$. Note that the morphism $\rho_{\@X}: h[\@X] \rightarrow \-{\@X}$ is definable. 

\begin{lemma}
Let $f: \@U \rightarrow \@X$ be a f-surjective smooth covering map of affine schemes over $k[[t]]$. Then $[\phi] \in I_{\-{\@X}}C(\@X \rightarrow \-{\@X})$ if and only if then $[f^*(\phi)] \in I_{\-{\@U}} C(\@U \rightarrow \-{\@U})$. If these conditions hold, then
\[
\mu_{\-{\@U}}([f^*(\phi)]) =  \-f^*(\mu_{\-{\@X}}([\phi]))  \text{.}
\] 
\end{lemma}

\begin{proof}
First note that $[\phi] \in I_{\-{\@X}}C(X \rightarrow \-{\@X})$ if and only if $[\phi] \in IC(X)$ (see \cite{CL}, Thm 10.1.1, part (A5) and Remark 14.2.3). Thus we simply wish to prove that $[\phi] \in IC(\@X)$ if and only if $f^*([\phi]) \in IC(\@U)$. 

First we note a consequence of the f-surjectivity of $f$ that we will use below. Since $f$ is f-surjective, so is $\-f$. In particular, every generic point of $\-{\@X}$ can be lifted to $\-{\@U}$. Thus there is an open dense subscheme $X_0 \subset \-{\@X}$ and a $\-{\@X}$-morphism $X_0 \rightarrow \-{\@U}$. Repeating the procedure for $\-{\@X} - X_0$, we see that $\-f$ has a definable (in the language of rings) section $s: \-{\@X} \rightarrow \-{\@U}$. 

The fact that $[\phi] \in IC(\@X)$ if and only if $f^*([\phi]) \in IC(\@U)$ is a consequence of (\cite{CL}, Thm 10.1.1, Part (A3)) which says that for $\alpha \in \@C(\@X)$ and $\beta \in IC(\@U)$, $\alpha f_!(\beta) \in IC(\@X)$ if and only if $f^*(\alpha) \beta \in IC(\@U)$ and that if these conditions are verified, then $f_!(f^*(\alpha)\beta) = \alpha f_!(\beta)$. 

We apply this with $\beta = [\mathbbm{1}_{\@U}]$ and $\alpha = \phi$. Since $f$ is a smooth morphism, it is easy to see that 
\[
f_!([\mathbbm{1}_{\@U}]) = \rho_{\@X}^*(\-f_{!}(\mathbbm{1}_{\-{\@U}})) \text{.} 
\]

Thus, $f^*(\phi) = f^*(\phi) \mathbbm{1}_{\@U} \in IC(\@U)$ if and only if $\rho_{\@X}^*(\-f_{!}(\mathbbm{1}_{\-{\@U}})) [\phi] \in IC(\@X)$. It is clear that if $[\phi] \in IC(\@X)$, then $\rho_{\@X}^*(\-f_{!}(\mathbbm{1}_{\-{\@U}}))[\phi] \in IC(\@X)$. To prove the converse, we may assume that $\phi \in \@C_+(X)$, the semiring of positive constructible motivic functions on $\@X$ (see \cite{CL}, Section 5). Let $\gamma = [\-{\@U} \backslash s(\-{\@X)}] \in K_0(RDef_{\-{\@X}})$ so that $\-f_{!}(\mathbbm{1}_{\-{\@U}}) = \gamma + [\-{\@X}]$. Thus we have 
\[
\rho_{\@X}^*(\-f_{!}(\mathbbm{1}_{\-{\@U}}))[\phi] = \gamma[\phi] + [\phi]
\]
and $\gamma[\phi] \in C_+(\@X)$, the semigroup positive constructible motivic Functions on $\@X$. Now apply (\cite{CL}, Thm 12.2.1) to conclude that $[\phi] \in IC(\@X)$. 

Finally, we prove the equality in the statement of the lemma. (For this part, it is not necessary to assume that $f$ is f-surjective.) To prove this, we may choose a suitable cover $\{\@U_i\}_i$ of $\@U$ by affine open subschemes $\@U_i$ and prove the result after replacing $\@X$ by $f(\@U_i)$ and $\@U$ by $\@U_i$. By choosing the pieces of the cover to be sufficiently small, we may assume that $f$ factors as 
\[
\@U \stackrel{f^{\prime}}{\longrightarrow} \|A^{\dim(\@U/\@X)}_{\@X} \longrightarrow X \text{,}
\]
where $f^{\prime}$ is étale and the morphism $\|A^{\dim(\@U/\@X)}_{\@X} \rightarrow X$ is the projection. Thus it suffices to prove the result for étale maps and maps of the form $\|A^{\dim(\@U/\@X)}_{\@X} \rightarrow X$. The latter case is trivial. Thus we now consider the case when $f$ is étale. But in this case, it is clear that $g: \@U \rightarrow \-{\@U} \times_{\-{\@X}} \@X$ is a definable isomorphism with $\rm{ordjac}(g) = 0$ (\cite{CL}, Section 8 and Thm. 12.1.1).
\end{proof}

This leads us to the following definition:

\begin{definition}
Let $\@X$ be a sft-Artin stack over $k[[t]]$ and let $f: \@U \rightarrow \@X$ be a f-surjective smooth atlas with $\@U$ being an affine scheme of finite type over $k[[t]]$. Then we define
\[
IC(\@X) = \{ [\phi] \in C(\@X)| [f^*(\phi)] \in IC(\@U)\} \text{.}
\]
\end{definition}
It follows from the preceding lemma that $IC(\@X)$ is independent of the choice of $f$. 

\begin{lemma}
\label{motivic-measure-computation}
Let $\@X$ be a sft-Artin stack over $k[[t]]$. Let $[\alpha] \in IC(\@X)$. Let $f: \@U \rightarrow \@X$ be a smooth covering map where $\@U$ is an affine scheme of finite type over $k[[t]]$. Let $\phi$ be the constructible function $\mu_{\-{\@U}}([f^*(\alpha)]) \in \@C(\-{\@U})$. Let $Z$ be any affine scheme of finite type over $k$, let $x: Z \rightarrow \-{\@X}$ by any morphism  and let $u: Z \rightarrow \-{\@U}$ be such that $\-p \circ u \cong x$. Then the element 
\[
u^*(\phi) \in \@C_{+}(Z)
\]
depends only on $x$ and $\alpha$, i.e. it is independent of the choice of $\@U$ and $u$.  
\end{lemma}

\begin{proof}
We may assume that $Z$ is irreducible. Suppose $f_1: \@U_1 \rightarrow \@X$ and $f_2: \@U_2 \rightarrow \@X$ are two choices for the atlas as mentioned in the statement of the lemma. Let $\phi_i := \mu_{\-{\@U}}([f_i^*(\alpha)]) \in \@C(\-{\@U_i})$. Let $u_i: Z \rightarrow \-{\@U}_i$ be a lift of $x$ to $\-{\@U_i}$ for $i = 1,2$. We wish to prove that $u_i^*(\phi_i) \in \@C_{+}(Z)$ is the same element for $i = 1,2$. 

There exists a morphism $v: Z \rightarrow \-{\@U_1 \times_{\@X} \@U_2}$ which lifts $u_1$ and $u_2$. Let $\eta$ denote the generic point of $Z$. Choose a smooth atlas $\@U_3 \rightarrow \@U_1 \times^h_{\@X} \@U_2$ such that $\@U_3$ is an affine scheme of finite type over $k[[t]]$ and such that $v|_{\eta}$ can be lifted to $U_3$. Then there is an open subscheme $Z_0 \subset Z$ such that the morphisms $u_i|_{Z^{\prime}}$ have a common lift to $U_3$. It will suffice to prove the result with $Z_0$ in place of $Z$. Indeed, if we can do this, we can repeat the procedure for $Z_1 = Z \backslash Z_0$. Proceeding in this manner, we can get the required result since $Z$ is noetherian. Thus we may assume that the morphisms $u_i: Z \rightarrow U_i$ for $i=1,2$ have a common lift $u_3:Z \rightarrow U_3$.  

Let $g_i: U_3 \rightarrow U_i$ for $i=1,2$ be the obvious maps and let $f_3:= f_1 \circ g_1 = f_2 \circ g_2$. Let $\phi_3 := \mu_{\-{\@U}}([f_3^*(\alpha)]) \in \@C(\-{\@U_3})$. Then it is clearly enough to show that $u_1^*(\phi_1) = u_3^*(\phi) = u_2^*(\phi_2)$. In other words, it is enough to prove the lemma assuming that $\@X$ is an affine scheme. 

Thus now assume that $\@X$ is an affine scheme and that $f: \@U \rightarrow \@X$ is a smooth covering map of affine schemes. It will suffice to prove the result when $Z = \-{\@U}$. In other words, we wish to prove that 
\[
\phi = \-f^*(\mu_{\-{\@X}}([\alpha])) \text{.}
\]
This is precisely the content of the previous lemma.
\end{proof}

Thus, with the notation of the above lemma, we have a rule which, for every affine scheme $Z$ of finite type over $k$ and a morphism $x: Z \rightarrow \-{\@X}$, assigns a constructible function $\phi_x \in \@C(Z)$ in a coherent manner. This leads us to define the following:

\begin{definition}
Let $X$ be a sft-Artin stack over $k$. A \textbf{constructible pseudo-function on $X$} is a rule $\phi$ which for every affine scheme $Z$ of finite type over $k$ and morphism $x: Z \rightarrow X$ assigns an element $\phi_x \in \@C(Z)$ such that if we have a commutative (upto homotopy) diagram
\[
\xymatrix{
Z_2 \ar[r]^{\alpha} \ar[dr]_{x_2} & Z_1 \ar[d]^{x_1}\\
& X
}
\] 
then $\alpha^*(\phi_{x_1}) = \phi_{x_2}$. 

We abuse notation and write $x^*(\phi)$ instead of $\phi_x$ even though $\phi$ is not a constructible function on $X$ in the usual sense. The constructible functions on $X$ clearly from a ring which we denote by $\@C(X)^{ps}$. 
\end{definition}

\begin{lemma}
\label{defining-pseudo-constructible-functions}
Suppose $X$ is a sft-Artin stack over $k$ and $\phi \in \@C(X)^{ps}$. Let $f: U \rightarrow X$ be an f-surjective morphism (not necessarily smooth) with $U$ being an affine scheme of finite type over $k$. Then $\phi$ is completely determined by $f^*(\phi)$. 
\end{lemma}
\begin{proof}
Indeed, suppose $Z$ is an affine scheme of finite type over $k$ and $x \rightarrow X$ is a morphism. We wish to compute $x^*(\phi)$. We may assume that $Z$ is irreducible. If $\eta$ is its generic point, then we can lift $x|_{\eta}$ to $U$. Thus there exists an open subscheme $Z_0 \subset Z$ such that $x|_{Z_0}$ can be lifted to $U$. Then, as in the proof of Lemma \ref{motivic-measure-computation}, the fact that $Z$ is noetherian allows us to compute $\phi_x$. It is easy to see that the calculation does not depend on the choice of $Z_0$. 
\end{proof}

Using Lemma \ref{motivic-measure-computation} and Lemma \ref{defining-pseudo-constructible-functions}, we may define the following:

\begin{definition}
Let $\@X$ is a sft-Artin stack over $k[[t]]$. Let $f: \@U \rightarrow \@X$ be an f-surjective atlas with $\@U$ being an affine scheme  of finite type over $k[[t]]$. We denote by $\mu_{\-{\@X}}: IC(\@X) \rightarrow \@C(\-{\@X})^{ps}$ the group homomorphism which maps an element $[\alpha] \in IC(\@X)$ to the element of $\@C(\-{\@X})^{ps}$ determined by the element $\mu_{\-{\@U}}([f^*(\alpha)]) \in \@C(\-{\@U})$.

\end{definition}

At this point, we pause for some remarks to clarify some of the choices we have made in this construction.

\begin{remark}
\label{explanation-pseudo-constructible-functions}
Obviously, the above definition is not quite satisfactory since we expect $\mu_{\-{\@X}}$ to take values in the ring $\@C(\-{\@X})$. Every element $\phi \in \@C(\-{\@X})$ gives rise to an element of $\@C(\-{\@X})^{ps}$ which is represented by the element $f^*(\phi) \in \@C(\-{\@U})$. This gives us a ring homomorphism $\@C(\-{\@X}) \rightarrow \@C(\-{\@X})^{ps}$. It is not clear whether this is an isomorphism except in when $\@X$ is an algebraic space (in which case the fact that this homomorphism is an isomorphism can be proved by choosing a ``definable section" for the morphism $\-{\@U} \rightarrow \-{\@X}$). Hence our usage of the ring $\@C(\-{\@X})^{ps}$ is essentially a compromise. We note that when $\@X$ is an affine scheme, we are using the symbol $\mu_{\-{\@X}}$ to denote two maps, one with domain $\@C(\-{\@X})^{ps}$ and the other with domain $\@C(\-{\@X})$. However, these two maps can be identified via the isomorphism $\@C(\-{\@X}) \rightarrow \@C(\-{\@X})^{ps}$ and so this is a harmless abuse of notation. 

\end{remark}

\begin{remark}
\label{explanation-modified-measure}
A disadvantage of having to use $\@C(\-{\@X})^{ps}$ instead of $\@C(\-{\@X})$ we are not able to compute motivic measure in the usual sense. Indeed, if we were able to use $\@C(\-{\@X})$, then we would be able to use the obvious pushforward from the ring $\@C(\-{\@X})$ into the ring $\@C(\Spec(k[[t]] \times \Spec(k))^{st}$ which would give us motivic measure in the usual sense. Indeed, this would be necessary if we wished to compare motivic measure on different stacks. 

One way to rectify this is to construct motivic measure on sft-Artin in a manner analogous to the construction for varieties in \cite{DL}. In the context of geometric motivic integration, this procedure has been carried out in \cite{Ba}. This can be done, but we do not do this here since the arguments are very similar to those in \cite{Ba}. 
\end{remark}

\section{Specialization to $p$-adic integration}

We now return to the problem of applying the theory of motivic integration to $p$-adic integration on sft-Artin stacks. We briefly recall the process of specialization to $p$-adic integration using (\cite{CL2}) as our reference. 

Let $k$ be a number field with $\@O$ being its ring of integers. Recall the terminology from (\cite{CL2}, Section 9) that the language $\@L_{\@O}$ is the language $\Ldp$ with $\Val$-type constants added for the elements of $\@O[[t]]$ and $\Res$-type constants added for the elements of $\@O$. Then we may consider definable subassignments and constructible functions that are definable in the language $\@L_{\@O}$. For a subassignment $S$ definable in $\@L_{\@O}$, we have subrings $K_0(\RDef_S, \@L_{\@O})$ and $\@C(S, \@L_{\@O})$ respectively.  

Let $K$ be a $p$-adic completion of $k$ with valuation ring $R_K$ and residue field $k_K$. Let $\omega_K$ be a uniformizing parameter in $R_K$. Then one has a $\@O$-algebra homomorphism $\lambda_{\@O,K}: \@O[[t]] \rightarrow K$ defined by 
\[
\lambda_{\@O,K}(\sum_{i \geq 0} a_i t^i) = \sum_{i \geq 0} a_i \omega^i \text{.}
\]
Also, for every $\alpha$ in $\@O$, let $\-{\alpha}$ denote the image of $\alpha$ under the quotient map $\@O \rightarrow k_K$. In a $\@L_{\@O}$ formula $\phi$, if we interpret every $\Val$-type constant $a \in \@O[[t]]$ as $\lambda_{\@O,K}(a) \in K$, and every $\Res$-type constant $\alpha \in \@O$ as $\-{\alpha} \in k_K$, then $\phi$ defines a subset $\phi_K$ of $R_K^m \times k_K^n \times \|Z^r$ for some non-negative integers $n,m,r$. We recall (\cite{CL2} or \cite{DL}) that if two formulas $\phi$ and $\psi$ define the same subassignment $S$ of $\|A^{m}_{k[[t]]} \times \|A^n_k \times \|Z^r$ then the subsets $\phi_K$ and $\psi_K$ of $R_K^m \times k_K^n \times \|Z^r$ are equal for almost all choices of $K$. One abuses notation to denote this set by $S_K$. Similarly, definable morphisms $f: S \rightarrow T$ give functions $f_K: S_K \rightarrow T_K$ for almost all $K$. 

Similarly, a constructible function on $S$ can be interpreted to give a function from $S_K$ into $\|Q$ for almost all $K$. For this we first interpret the elements of $K_0(\RDef_S, \@L_{\@O})$ and then the elements of $\@P(S)$. 

If $\phi \in K_0(\RDef_S, \@L_{\@O})$ is such that it is represented by $[\pi: W \rightarrow S$ with $W \in S \times \|A^n_k$. Then $W_K$ is in $S_K \times k_K^n$ and we have the projection map $\pi_K: W_K \rightarrow S_K$. Then we define the function $\phi_K$ on $X_K$ by $\phi_K(x):= |\pi_K^{-1}(x)|$. Then one extends this construction by linearity to the whole of $K_0(\RDef_S, \@L_{\@O})$. 

To interpret the elements of $\@P(S)$ over $K$, we express such an element $\phi$ in terms of $\|L$ and definable functions $\alpha: S \rightarrow \|Z$. Then we interpret $\|L$ as $q_K = |k_K|$ and $\alpha$ as a function $\alpha_K: S_K \rightarrow \|Q$ which is well-defined for almost all $K$. Now, we can interpret the elements of $\@C(S, \@L_{\@O})$ as functions on $S_K$ by tensoring. 

Now let $X$ be an sft-Artin stack over $\@O$. Let $X_t:= X \times_{\Spec(\@O)} \Spec(k[[t]])$ which is a sft-Artin stack over $k[[t]]$. We claim that a definable subassignment $S$ on $X_t$ defines a subset $S_K$ of $\pi_0(X(R_K))$ for almost all $K$. Indeed, we choose a f-surjective smooth atlas $f: U \rightarrow X$ with $U$ being an affine scheme of finite type over $\@O[[t]]$. Let $V \rightarrow U \times_X^h X$ be an f-surjective smooth atlas where $V$ is an affine scheme of finite type over $\@O$. Let $g_1, g_2: V \rightarrow U$ be the maps obtained by composing $V \rightarrow U \times^h_X U$ with the two projections $U \times_X^h U \rightarrow U$. Then $f^{-1}(S)$ is a definable subassignment on $U_{t}$ which defines a subset given by a formula $\phi$. Let $\psi_i$ be the pullback of the formula $\phi$ via $g_i$ for $i=1,2$. Clearly, $\psi_1$ and $\psi_2$ define the same subassignment on $V_{t}$. The formula $\phi$ defines a subset $\phi_K$ of $U(R_K)$. This is the preimage of a subset of $\pi_0(X(R_K))$ if and only if the subsets $(\psi_1)_K$ and $(\psi_2)_K$ are equal. But we know that this is true for almost all $K$. Thus we see that a definable subassignment on $X_t$ defines a subset $S_K$ of $\pi_0(X(R_K))$ for almost all $K$. By similar arguments, one can interpret constructible functions $\phi$ on $S_K$ to give functions $S_K \rightarrow \|Q$ for almost all $K$. By similar arguments, one can show that a constructible function $\phi \in \@C(S)$ defines a function $\phi_K: S_K \rightarrow \|Q$ for almost all $K$. 

For any affine scheme $T$ of finite type over $\@O$, we note that the set $(\-{T_t})_K$ is simply the set $T(k_K)$ of $k_K$ valued points on $T$. Indeed, $T$ can be defined as a closed subscheme of $\|A^n_{\@O}$ cut out by polynomials with coefficients in $\@O$. The same polynomials can be used to define $T_t$ as a closed subscheme of $\|A^n_{\@O[[t]]}$ and $\-{T_t}$ as a closed subscheme of $\|A^n_k$. Interpreting the coefficients of those polynomials as elements in $k_K$ via the map $\alpha \mapsto \-{\alpha}$, we see that 
\[
(\-{T_t})_K = (T \times_{\Spec(\@O} \Spec(k_K))(k_K) = T(k_K) \text{.}
\]

Let $\phi \in \@C(\-{X_t}, \@L_{\@O})^{ps}$. Then with $U$ as above, $f^*(\phi) \in \@C(\-{U_t}, \@L_{\@O})$. Then by the above arguments we can define a function $(f^*(\phi))_K$  on the set $(\-{U_t})_K = U(k_K)$ for almost all $K$.  

\begin{claim}
The function $(f^*(\phi))_K$ is constant on the fibres of the map $U(k_K) \rightarrow \pi_0(X(k_K))$. 
\end{claim}   
\begin{proof}
Using the notation we set up above, we look at the functions $g_1^* \circ f^* (\phi)$ and $g_2^* \circ f^* (\phi)$ on $\-{V_t}$. By the definition of a constructible pseudo-function, these functions are equal. Thus 
\begin{eqnarray*}
(g_1)_K^* (f^* (\phi))_K & = & (g_1^* \circ f^*(\phi)))_K \\
                       & = & (g_2^* \circ f^*(\phi)))_K \\
                       & = & (g_2)_K^* (f^* (\phi))_K
\end{eqnarray*}
for almost all $K$. If $u_1$ and $u_2$ are two points of $U(k_K)$ which lie in the same fibre of $U(k_K) \rightarrow \pi_0(X(k_K))$, then they have a common lift $v$ to $V(k_K)$. This proves our claim. 
\end{proof}

Given any $\phi \in  \@C(\-{X_t}, \@L_{\@O})^{ps}$, we now define its evaluation $\gamma_{K}(\phi) \in \|Q$ which is defined for almost all $K$. To do this, for each $x \in \pi_0(X(k_K))$, we choose an arbitrary $u_x$ in the fibre of $U(k_K) \rightarrow \pi_0(X(k_K))$ over $x$. Then we define
\[
\gamma_K(\phi) = \sum_{x \in \pi_0(X(k_K))} (f^*(\phi))_K(u_x) \cdot \#\{x\} 
\]
where $\#$ is the counting measure we defined in Section \ref{p-adic-measure-definition}, equation (\ref{counting-measure}). By the claim we proved above, this is independent of the choice of $u_x$ for almost all $K$. 

\begin{theorem}
\label{specialization}
Let $X$ be a sft-Artin stack over $\@O$ with $\dim(X/\@O)=d$. let $\phi \in IC^d(X_t,\@L_{\@O})$. Then for almost all $K$, $\phi_K$ is integrable over $\pi_0(X(R_K))$ and 
\[
\gamma_K(\mu_{\-{X_t}}(\phi)) = \int_{\pi_0(X(R_K))} \phi_K d \mu_d \text{.}
\]
\end{theorem}

\begin{proof}
Let $f: U \rightarrow X$ be an f-surjective smooth atlas with $U$ being an affine scheme over $\@O$ with $\dim(U/\@O)=e$. Then we know that $f_t^*(\phi)$ is in $IC^e(U_t)$. From (\cite{CL2}, Thm. 9.1.5), it follows that for any $u \in U(k_K)$, 
\[
(\mu_{\-{U_t}}(f_t^*(\phi)))_K(u) = \int_{\tau_0^{-1}(u)} \phi_K d \mu_d 
\]
where $\tau_0$ is the truncation map as defined in section \ref{p-adic-measure-on-stacks}. If $x \in \pi_0(X(k_K))$ and $U_x:= U \times_{X,x} \Spec(k_K)$ and $f^{-1}(x) := \{u \in U(k_K)| f(u)=x \}$, we know from Lemma \ref{counting-by-fibres} that $\#f^{-1}(x) = (\#U_x(k_K)) \cdot (\#\{x \})$. Also, by the above claim, $(\mu_{\-{U_t}}(f_t^*(\phi)))_K(u)$ is constant as $u$ varies through $f^{-1}(x)$. 
Thus
\begin{eqnarray*}
\gamma_K(\mu_{\-{X_t}}(\phi)) & = & \sum_{x \in \pi_0(X(k_K))} (\mu_{\-{U_t}}(f_t^*(\phi)))_K(u_x) \cdot \#\{x\} \\
               & = & \sum_{x \in \pi_0(X(k_K))} \left(\frac{1}{\#U_x(k_K)} \right)(\mu_{\-{U_t}}(f_t^*(\phi)))_K(u_x) \cdot \#f^{-1}(x) \\
               & = & \sum_{u \in U(k_K)}\left(\frac{1}{\#U_x(k_K)} \right)\mu_{\-{U_t}}(f_t^*(\phi))_K(u) \\
               & = & \sum_{u \in U(k_K)}\left(\frac{1}{\#U_x(k_K)} \right)  \int_{\tau_0^{-1}(u)} \phi_K d \mu_d \\
               & = & \int_{\pi_0(X(R_K))} \phi_K d \mu_d
\end{eqnarray*}
as required.
\end{proof}

Let $\@C(\-{X_t}, \@L_{\@O})^{ps}[[T]]_{\rm{rat}}$ be the subring of $\@C(\-{X_t}, \@L_{\@O})^{ps}[[T]]$ generated by $\@C(\-{X_t}, \@L_{\@O})^{ps}[T]$ and the set $(1-\|L^aT^b)^{-1}$ where $a \in \|Z$ and $b \in \|N \backslash \{0\}$. For any $R(T) \in  \@C(\-{X_t}, \@L_{\@O})^{ps}[T]$ we can define the element $\gamma_K(R(T)) \in \|Q[T]$ for almost all $K$ by applying $\gamma_K$ to the coefficients of $R(T)$. Then by mapping $(1-\|L^aT^b)^{-1}$ to $(1-q_K^aT^b)^{-1}$, we can define element $\gamma_K(P(T)) \in \|Q(T)$ for almost all $K$. 

\begin{theorem}
Let $X$ be an sft-Artin stack over $\@O$. Then there exists an element $P_X(T) \in \@C(\-{X_t}, \@L_{\@O})^{ps}[[T]]_{\rm{rat}}$ such that $\gamma_K(P_X(T)) = P_{X_K}(T)$  for almost all $K$. 
\end{theorem}

\begin{proof}
By the comments preceding Lemma \ref{Q_X}, it suffices to prove the theorem with $Q_X$ in place of $P_X$. In other words, we wish to show that there exists an element $Q_X(T) \in \@C(\-{X_t}, \@L_{\@O})^{ps}[[T]]_{\rm{rat}}$ such that $\gamma_K(Q_X(T)) = Q_{X_K}(T)$ for almost all $K$ (see Definition \ref{Q_X-definition} for the definition of the power series $Q_{X_K}(T)$). 

Let $f: U \rightarrow X$ be an f-surjective smooth atlas with $U$ being an affine scheme over $\@O$. Let $M^U_n$ be the definable subassignment on $U_t$ defined by the condition $u \notin U_{sing} \mod(t^{n+1})$. Let $M_n^X$ be the image of $M^U_n$ in $X_t$. Then it is easy to see that $M_n^U = f_t^{-1}(M^U_n)$. We define $\phi_n^X$ (resp. $\phi^U_n$) to be the characteristic function of $M^X_n$ (resp. $M^U_n$). Then it is easy to see from Lemma \ref{Q_X} and Theorem \ref{specialization} that if we define $Q_X(T)$ by the formula 
\[
Q_X(T):= \sum_{n=0}^{\infty} \|L^{nd} \mu_{\-{X_t}}(\phi^X_n) T^n \text{,}
\]
then $\gamma_K(Q_X(T)) = Q_{X_K}(T)$ for almost all $K$. Also, $Q_X(T) \in \@C(\-{X_t}, \@L_{\@O})^{ps}[[T]]$ is represented by 
\[
f_t^*(Q_U(T))= \sum_{n=0}^{\infty} \|L^{nd} \mu_{\-{U_t}}(\phi^U_n) T^n \text{,}
\]
in $\@C(\-{U_t}, \@L_{\@O})[[T]]$. By (\cite{CL}, Theorem 14.4.1), it is known that $f_t^*(Q_U(T)) \in \@C(\-{U_t}, \@L_{\@O})[[T]]_{\rm{rat}}$. This proves the result. 
\end{proof}

\end{document}